\documentclass{amsart}
\usepackage{amssymb}
\usepackage{amsfonts}
\usepackage{amsmath}
\usepackage{graphicx}%
\usepackage{lastpage}
\usepackage[utf8]{inputenc}
\usepackage[legalpaper,bookmarks=true,colorlinks=true,linkcolor=blue,citecolor=blue]{hyperref}
\usepackage{fancyhdr}
\usepackage{color}
\usepackage[mathlines]{lineno}
\usepackage{multirow}
\usepackage{caption}
\usepackage{subcaption}
\usepackage{lscape}
\usepackage{epsfig}
\usepackage{natbib}
\usepackage{float}
\usepackage{slashbox}
 \usepackage{tikz}

\theoremstyle{plain}
\newtheorem{example}{Example}

\newtheorem{proposition}{Proposition}

\numberwithin{equation}{section}

\author{ BA Amadou Diadie}
\email{ba.amadou-diadie@ugb.edu.sn}
\author{LO Gane Samb}
\email{gane-samb.lo@ugb.edu.sn}
\begin{document}
%\title[Short title]{Possibly very long title}
\title[Uncertainty and Inaccuracy measure estimation]{Non parametric estimation of residual-past entropy, mean residual-past lifetime, residual-past inaccuracy measure and asymptotic limits}

\begin{abstract}
In the present work, we provide the asymptotic behavior of %theory pose non parametric estimators for the 
residual-past entropy, of the mean residual-past lifetime distribution and of the residual-past inaccuracy measure. We are interested in these measures of uncertainty in the discrete case. Almost sure rates of convergence and asymptotic normality results are established. Our theoretical results are validated by simulations.\\ 
\end{abstract}

\maketitle
\section{Introduction}
\subsection{Motivation}

\noindent Let $X$ be a finite discrete random  describing the lifetime of a component/system and defined on a probability space $(\Omega,\mathcal{A},\mathbb{P})$. Suppose %, taking values in the finite countable space
\\ 
$X(\Omega)=\{x_1,\cdots,x_r\}\, \ \  (0<x_1<\cdots<x_r)$, with probability mass function (\textit{p.m.f.}) $ p_{j}=\mathbb{P}(X= x_j),\ \ j\in J=[1,r]$ and denote $P$ and $\overline{P}$ the corresponding cumulative distribution and survival functions defined respectively by :
\begin{eqnarray*}\label{pkovpk}
&&P(x)=\mathbb{P}(X\in(-\infty,x]) =\sum_{j\in J,\,x_j\leq x}p_j\\
\text{and} \ \ 
 && \overline{P}(x)=\mathbb{P}(X\in (x,+\infty))=\sum_{j\in J,\,x_j> x}p_j
 \end{eqnarray*} for any $x\in  \mathbb{R}% j\in J
 $. A classical measure of uncertainty
for the random variable $X$ is the \textit{Shannon entropy} also known as the Shannon information measure, defined as (see \cite{shan2})  
\begin{equation}\label{shn}
\mathcal{E}_{Sh}(X)=-\sum_{J\in J}p_j\log p_j
\end{equation}
where $"\log"$ stands for the natural logarithm. \\

\bigskip \noindent In the literature, the reliability and the information theory are used to study the behaviour
of a component/system. Given that at age $x_j$, the component has survived up to age $x_j$ or has been found failing at age $x_j$, $\mathcal{E}_{Sh}(X)$ is no longer useful for measuring the uncertainty about the remaining lifetime or about the past lifetime since the age should be taken into  account (see \cite{ebrah}).\\
 
\noindent  Instead, many others measures of uncertainty was defined such as 
the \textit{residual entropy}, \textit{past entropy}, \textit{cumulative residual entropy}, \textit{cumulative past entropy}, \textit{cumulative residual entropy}, \textit{cumulative past entropy}, etc. Their importance can be seen through their
appearance in several important theorems of information theory such as reliability 
engineering, survival analysis, demography, actuarial study and 
others.\\

\noindent These measures of uncertainty are defined in the continuous setting, but 
there are many situations where a continuous time is inappropriate for describing the lifetime of devices and other systems. For example 
the lifetime of many devices in industry, such as switches and mechanical tools, depends essentially on the number of times that they are turned on and off or the number of shocks they receive. In such cases, the time to failure is often more appropriately represented by the number of times they are used before they fail, which is a discrete random variable.\\

\bigskip \noindent Typically, ‘\textit{lifetime}’ refers to human life length, the 
life span of a device before it fails,
 the survival time of a patient 
with  serious  disease  from  the  date  of  diagnosis  or  major 
treatment or the duration an individual remains married, the 
durations  of  coalitions, the length of time to complete a PhD degree, the duration an individual remains 
unemployed,  the  duration  an  individual  stays  in  an employment, the duration of a  war, the length a leader stays in power,   etc.\\

\bigskip  \noindent
 For $ j\in [1,r-1]$, the  random variable $X^{(j)}=[X-x_j/X>x_j]$ describes the remaining lifetime of the component, given that it has survived up to time $x_j$. Whereas
  the random variable $X_{(j)}=[x_j-X|X\leq x_j]$,  for $j\in [1,r]$, describes the past  lifetime of the component given that at time $x_j$ it has been found failed.
  \\
  
  \bigskip
\noindent 
 \textit{(a)} The
 \textit{discrete residual entropy} of the random lifetime $X$ at time $x_j,\,j\in [1,r-1]$ %j\in\{0,\cdots,r-2\}
is 
\begin{equation}\label{shres}
\mathcal{R}_X(x_j):=\mathcal{E}_{Sh}(X^{(j)})=-\sum_{k=j}^{r} \frac{p_k}{\overline{P}(x_j)}\log \frac{p_k}{\overline{P}(x_j)}.
\end{equation}

\noindent  $\mathcal{R}_X(x_j)$ measures the uncertainty contained in $X-x_j$ % (called the remaining lifetime)  
 given that$X>x_j$.\\

\noindent \textit{(b)} The  \textit{discrete past entropy} of the random lifetime $X$ at time $x_j,\,j\in [1,r]$ is 
 \begin{equation}\label{shpas}
\mathcal{P}_X(x_j):=\mathcal{E}_{Sh}(X_{(j)})=-\sum_{k=1}^j \frac{p_k}{P(x_j)}\log \left( \frac{p_k}{P(x_j)}\right).
\end{equation}

\noindent $\mathcal{P}_X(x_j)$ % is the Shannon entropy of $X_{(j)}$.  It 
measures the uncertainty of $x_j-X$  given that 
 $X\leq x_j$.\\
  
\noindent Obviously, we have $$\mathcal{P}_X(x_r)=\mathcal{E}_{Sh}(X)
.$$

\bigskip \noindent The two following measures of uncertainty measure the information contained in the survival function and in the cumulative distribution function of $X$.  \\ 

\noindent \textit{(c)}  The \textit{Discrete Cumulative residual entropy} of $X$ is defined by 
\begin{eqnarray}\label{crex}
\mathcal{CR}_{X}&=&-\sum_{j=1}^{r-1}\overline{P}(x_j)\log \overline{P}(x_j).
\end{eqnarray}

\noindent \textit{(d)}  The \textit{Discrete Cumulative past entropy} of $X$ is defined by
\begin{eqnarray}\label{cpe}
\mathcal{CP}_X&=&-\sum_{j=1}^{r}P(x_j)\log P(x_j).
\end{eqnarray}

\noindent $\mathcal{CP}_X$ is useful to measure information on the inactivity 
time of a  system, being appropriate for the systems whose uncertainty  is related to the 
past.\\

\noindent Two other important measures are 

\bigskip \noindent \textit{(e)}  The \textit{mean residual lifetime} of $X$ at time $x_j$ is (see \cite{merv})
\begin{equation}
\mu_R(x_j)=\mathbb{E}\left(X^{(j)}\right)= %\frac{1}{\overline{P}(x_j)}\sum_{k=j}^{r-1}\overline{P}(x_k)=
\frac{1}{\overline{P}(x_j)}\sum_{k=j}^{r-1}\overline{P}(x_k),\ \ j\in [1,r-1].
\end{equation}

\bigskip \noindent \textit{(f)} The \textit{mean inactivity (past) lifetime}  of $X$ at time $x_j$ is 
\begin{eqnarray*}
\mu_P(x_j)=\mathbb{E}\left( X_{(j)}\right)= \frac{1}{P(x_j)}\sum_{k=1}^jP(x_k),\ \ j\in [1,r].
\end{eqnarray*}
 It is of interest in many fields such as reliability, survival
analysis, actuarial studies, etc. \\

\bigskip \noindent Another generalization of the Shannon entropy for measuring the error in experimental outcomes is the \textit{inaccuracy measure.} Suppose that $X$ is the actual random variable corresponding to the observations and $Y$ is the random variable assigned by the experimenter with \textit{p.m.f.}'s $\textbf{q}=(q_j)_{j\in J}$.\\

\bigskip \noindent \textit{(g)} The \textit{discrete inaccuracy measure} of $X$ and $Y$ is defined  as (see \cite{kerr})
\begin{equation}\label{defg}
K_{(X,Y)}=-\sum_{j\in J} p_j\log q_j.
\end{equation}

\noindent It has applications  in the statistical inference, estimation and coding theory.\\

 \noindent (g) \textit{The discrete residual inaccuracy measure} of $X$ and $Y$ at time $x_j$, $j\in [1,r-1]$ is defined by 

\begin{equation}
\mathcal{R}_{(X,Y)}(x_j)=-\sum_{k=j}^{r} \frac{p_k}{\overline{P}(x_j)}\log \frac{q_k}{\overline{Q}(x_j)},
\end{equation}
where
\begin{eqnarray*}
\overline{Q}(x)=1-Q(x)\ \ \text{with}\ \ Q(x)=\mathbb{P}(Y\in(-\infty,x]),
 \end{eqnarray*} for any $x\in  \mathbb{R}% j\in J
 $.\\
 
\bigskip \noindent (h)  \textit{The discrete past inaccuracy measure} of $X$ and $Y$ at time $x_j,$ $j\in [1,r]$ is defined by 
\begin{equation}
\mathcal{P}_{(X,Y)}(x_j)=-\sum_{k=1}^j \frac{p_k}{P(x_j)}\log \frac{q_k}{Q(x_j)}.
\end{equation}

\noindent It's clear  $\mathcal{P}_{(X,Y)}(x_r)= K_{(X,Y)}$.
 \\
 
 \bigskip \noindent Analogous to $\mathcal{CR}_X$ and $\mathcal{CP}_X$ the two following information measures can be considered.\\
 
  \bigskip \noindent (i) The \textit{discrete cumulative residual inaccuracy of $X$ and $Y$} is defined as
\begin{equation}\label{crin}
\mathcal{CR}_{(X,Y)}=-\sum_{j=1}^{r-1} \overline{P}(x_j)\log \overline{Q}(x_j).
\end{equation} 
  
\bigskip  \noindent (j) The \textit{discrete cumulative past inaccuracy of $X$ and $Y$} is defined as
 \begin{equation}\label{cpin}
\mathcal{CP}_{(X,Y)}=-\sum_{j=1}^{r} P(x_j)\log Q(x_j),
\end{equation} 

\bigskip 
\noindent In particular, when the two distributions $\textbf{p}$ and $\textbf{q}$ coincide, then 
\begin{eqnarray}\label{sinkul}
&&K_{(X,Y)}=\mathcal{E}_{Sh}(X)=\mathcal{P}_X(x_r),\\
\nonumber && \mathcal{R}_{(X,Y)}(x_j)=\mathcal{R}_{X}(x_j),\ \ \ \  \mathcal{P}_{(X,Y)}(x_j)=\mathcal{P}_{X}(x_j), \ \ \ \text{for any} \ \ j\in [1,r-1],\ \ \\
\nonumber &&\mathcal{CR}_{(X,Y)}=\mathcal{R}_{X},\ \ \ \ \text{and}\ \ \ \ \mathcal{CP}_{(X,Y)}=\mathcal{P}_{X}.
\end{eqnarray}

\bigskip \noindent A generalization of Formula \eqref{sinkul}
is the following
\begin{equation}
K_{(X,Y)}=\mathcal{E}_{Sh}(X)+\mathcal{D}_{KL}(X,Y),
\end{equation}

\bigskip \noindent 
where $\mathcal{D}_{KL}(X,Y)$ is the Kullback-Leibler measure of discrimination (see Kullback, 1959), hence the inaccuracy measure of $X$ and $Y$ represents the \textit{information lost} when $\textbf{q}$ is used to approximate $\textbf{p}$.\\

\bigskip \noindent Many other extensions of Shannon entropy was defined (see \cite{rao}, \cite{driss}, \cite{sunoj}, \cite{psar}, \cite{sati},
\cite{raje}, and \cite{kund}.)\\

\noindent From this small sample of information measures, we may give the following remark :
for both the residual and past entropies, we may have computation problems. So
without loss of generality, suppose that
 
 \begin{eqnarray}\label{BD}
&& p_j>0,\ \  q_j>0,\ \ \overline{P}(x_j)>0,
\ \ \text{and}\ \ \overline{Q}(x_j)>0,\ \ \forall \,j\in J.
 \end{eqnarray}

\bigskip \noindent If Assumption \eqref{BD} holds, we do not have to worry about summation problems, especially
for  residual/past entropies, in the computations arising in estimation
theories. This explains why Assumption \eqref{BD} is systematically used in a great number of
works in that topic, for example, in \cite{kris}, \cite{hall}, and recently in \cite{ba5} to cite a few.\\

\bigskip \noindent Before coming back to our measures of uncertainty estimation of interest, let highlight some important applications of them. Indeed,  residual/past entropies have many applications in different branches of sciences such as in reliability engineering, computer vision (\cite{mura}), survival 
analysis, image processing (\cite{zohr}), actuarial sciences (\cite{atha}), social sciences, biological systems, etc. The Inaccuracy measure, for the lifetime distribution based on data, plays important roles in reliability and survival analysis in connection with modeling and analysis of life time data (\cite{thap}, 
\cite{saeid}, etc). It has applications in statistical inference, estimation and coding theory.

\subsection{Previous work}
Most of papers focus on residual/past entropies and on  residual/past inaccuracy measures for lifetime distribution in the continuous setting.
\\

 \noindent \cite{raj} proposed  nonparametric estimators for the residual entropy function based on censored data and established asymptotic properties of the estimator under suitable regularity conditions. \cite{osman} derived some properties of the cumulative past entropy of the last order statistics. \cite{enc} proposed an estimation of cumulative past entropy for power function distribution. 
Some authors investigated the asymptotic behavior of the mean residual lifetime, let cite \cite{yang},  \cite{HalWel}, \cite{ba2}, etc.\\
 
\noindent 
\cite{saeid} proposed cumulative past inaccuracy measure in lower record values and studied the problem of estimating the cumulative measure of inaccuracy by means of the empirical
cumulative inaccuracy in lower record values.\\

 \noindent In this present work, we propose a non-parametric estimate of most uncertainty and inaccuracy measures in the discrete case and we examine their asymptotic properties.

	\subsection{Overview of the paper}
	The remainder of the paper is organised as follows: Given an i.i.d. sample of size $n$ and according to $\textbf{p}$, we define, in \textsc{S}ection \ref{estimpj}, estimates $p_n^{(j)}$ of the \textit{p.m.f'}s $p_j$ and 
	we 
	construct the plug-in estimators of the discrete entropies and inaccuracy measures, we already described.\\ 
 In \textsc{S}ection \ref{main-res}, we  establish almost sure convergence and asymptotic normality of the estimators. In \textsc{S}ection \ref{simul} we present some simulations confirming our results. Finally,  
in \textsc{S}ection \ref{conclusion}, we conclude.

\section{Estimation}
\label{estimpj} \noindent In this section, we construct estimate of pmf $p_j$ from a sample of i.i.d. random variables according to $\textbf{p}$ and construct the plug-in estimates of informations measures of uncertainty cited above. \\

\noindent Let $X$ be a random variable defined on the probability  space $(\Omega,\mathcal{A},\mathbb{P})$ and taking values $X(\Omega)=\{x_1,\cdots,x_r\}\, \ \  (0<x_1<\cdots<x_r)$, with \textit{p.m.f.}'s $\mathbf{p}=(p_j)_{1\leq j\leq r}$ i.e,

$$p_j=\mathbb{P}(X=x_j),\ \ \forall j\in J=[1,r].$$

\bigskip \noindent In general, the full probability distribution $\mathbf{p}=(p_j)_{j\in J}$  is not known
and, in particular, in many situations only sets from which to infer entropies are
available. \\
 In such a case, one could estimate the probability $p_i$ of each element $i$ to occur.\\
 
\noindent Let $X_1,\cdots,X_n$ be $n$ i.i.d. random variables according to $\textbf{p}$. For a given $j\in J$, define the easiest and most objective estimator of $p_j$, based on the i.i.d sample $ X_1,\cdots,X_n,$ by 
\begin{eqnarray}\label{pn}
 p_n^{(j)}&=&\frac{1}{n}\sum_{i=1}^n1_{x_j}(X_i),
\end{eqnarray}
where \\
\noindent  $1_{x_j}(X_i)=\begin{cases}
 1\ \ \text{if}\ \ X_i=x_j\\
 0\ \ \text{otherwise}
 \end{cases} $.%  and 
 \\

\noindent For a given $j\in J$, this empirical estimator $p_n^{(j)}$ of $ p_{j} $ is strongly consistent and asymptotically normal. Precisely, when $n$ tends to infinity,
  \begin{eqnarray} \label{pxn} 
  && p_n^{(j)}-  p_{j} \stackrel{a.s.}{\longrightarrow} 0\ \ 
 \text{and}\ \   \sqrt{n}(p_n^{(j)}-  p_{j})  \stackrel{\mathcal{D}}{\rightsquigarrow}Z_{p_j},
\end{eqnarray}  where $Z_{p_j}\stackrel{d }{\sim}\mathcal{N}(0, p_{j} (1- p_{j} )).$\\

\noindent These asymptotic properties derive from the law of large
numbers and central limit theorem. \\

\noindent Here and in the following $\stackrel{a.s.}{ \longrightarrow}$ means the \textit{almost sure convergence}, $\stackrel{\mathcal{D}}{ \rightsquigarrow}
$ the \textit{convergence in distribution}, and $\stackrel{d }{\sim}$ means the \textit{distributional equality}. \\

%\noindent 
% 
\bigskip \noindent Based on the i.i.d sample $ X_1,\cdots,X_n,$ according to $\textbf{p}$, estimators of the cumulative distribution $P$ and survival functions $\overline{P}$ are given respectively by   
 \begin{eqnarray}\label{pnn}
P_n(x)&=&\frac{1}{n}\sum_{i=1}^n  \textbf{1}_{(-\infty,x]}(X_i)\ \
\ \ \text{and}\ \  \ \overline{P}_n(x) =
 \frac{1}{n}\sum_{i=1}^n\textbf{1}_{(x,+\infty)}(X_i).
 \end{eqnarray}
 
 \bigskip \noindent Again, for a fixed $j\in [1,r-1]$, we have when $n$ tends to infinity,
  \begin{eqnarray}
\label{ovepxj} && \overline{P}_n(x_j)-\overline{P}(x_j) \stackrel{a.s.}{\longrightarrow} 0\\ 
 \ \  && \label{norovp} \sqrt{n}(\overline{P}_n(x_j)-\overline{P}(x_j) )  \stackrel{\mathcal{D}}{\rightsquigarrow}Z_{\overline{P}(x_j)},
\end{eqnarray}  where $Z_{\overline{P}(x_j)}\stackrel{d}{\sim} N\left(0,\overline{P}(x_j)(1-\overline{P}(x_j)\right)$.  
\\

\bigskip \noindent  For sake of simplicity, we introduce the following notations :
\begin{eqnarray}
\label{abcn}
a_n(p)&=&\sup_{j\in J}|\Delta_n(p_j)|, \ \
 a_n(P)=\sup_{j\in J}\left\vert \Delta_n(P(x_j)) \right\vert,\\
 \text{and}\ \ a_n(\overline{P})&=&\sup_{ j\in [1,r-1]}\left\vert \Delta_n(\overline{P}(x_j)) \right\vert
\end{eqnarray}        
where $\Delta_n(p_j)=p_n^{(j)}-p_j$ and $\Delta_n(P(x_j))%=P_n(x_j)-P(x_j)
$ and $\Delta_n(\overline{P}(x_j))%=\overline{P}_n(x_j)-\overline{P}(x_j).
$ 
are the similarly defined for $P(x_j)$ and for $\overline{P}(x_j)$.\\

 \noindent Hence, from \eqref{pxn} and from \eqref{ovepxj}, we have
 \begin{equation}\label{maxanp}
 \max(a_n(p),a_n(\overline{P})) \stackrel{a.s.}{\longrightarrow}0,\ \ \text{as}\ \ n\rightarrow+\infty.
 \end{equation}

\noindent Set
\begin{eqnarray*}
&& \delta_n(p_j)=\sqrt{n/p_j}\Delta_n(p_j),\,j\in J,\ \ \ \delta_n(P(x_j))=\sqrt{n/P(x_j)}\Delta_n(P(x_j)),\, \ j\in[1,r]\\
 \text{and}\ \  &&\delta_n(\overline{P}(x_j))=\sqrt{n/\overline{P}(x_j)}\Delta_n(\overline{P}(x_j)),\, \ j\in[1,r-1].
\end{eqnarray*} 

\bigskip \noindent We recall that, since for a fixed $j\in J,$ $np_n^{(j)}$ has a binomial distribution with parameters $n$  and success probability $p_j$, we have 
 \begin{equation*}
 \mathbb{E}\left( p_n^{(j)}\right)=p_j\ \ \text{and}\ \ \mathbb{V}(p_n^{(j)})=\frac{p_j(1-p_j)}{n}.
\end{equation*} 

\noindent And finally, by the asymptotic Gaussian limit of the multinomial law (see for example Chapter 1, Section 4 in  \cite{ips-wcia-ang})),  we have, as $n\rightarrow+\infty$,
\begin{eqnarray}
\label{pnj}&& \biggr( \delta_n(p_j), \ j\in D\biggr)
\stackrel{\mathcal{D}}{\rightsquigarrow }Z(\textbf{p}) \stackrel{d }{\sim}\mathcal{N}(0,\Sigma_\textbf{p}),\\
 \ \  && \biggr( \delta_n(\overline{P}(x_j)), \ j\in [0,r-1]\biggr)\stackrel{\mathcal{D}}{\rightsquigarrow }Z(\overline{ \textbf{P}})\stackrel{d }{\sim}\mathcal{N}(0,\Sigma_{\overline{ \textbf{P}}}),
 \label{ovpnj}
\end{eqnarray}where $Z(\textbf{p})= (Z_{p_j},j\in J)^t
$ \
and $Z(\overline{ \textbf{P}})=(Z_{\overline{P}(x_j)},j\in [1,r-1])^t
$ are two independent centered Gaussian random vectors of respectives dimension $r$ and $r-1$
  having respectively the following elements :
\begin{eqnarray}\label{vars1}
&&\left(\Sigma_\textbf{p}\right)_{(i,j)}=(1-p_j) \delta_{ij}-\sqrt{ p_ip_j} (1-\delta_{ij}), \ \ (i,j) \in D^2\\
&& \label{vars2}\left(\Sigma_{\overline{\textbf{P}}}\right)_{(i,j)}=(1-\overline{P}(x_j)) \delta_{ij}-\sqrt{\overline{P}(x_i)\overline{P}(x_j)}  (1-\delta_{ij}), \ \ (i,j) \in [1,r-1]^2,
\end{eqnarray} 
where $ \delta_{ij}=\begin{cases}
1\ \ \text{for}\ \ i=j\\
0\ \ \text{for}\ \ i\neq j
\end{cases}.$\\

\bigskip \noindent For a fixed $j\in [1,r-1]$, we have also 
 \begin{eqnarray}
&& \label{esovp} \mathbb{E}\left( \overline{P}_n(x_j)\right)=\overline{P}(x_j),  \ \ \mathbb{V}(\overline{P}_n(x_j))=\frac{\overline{P}(x_j)(1-\overline{P}(x_j))}{n}.
 \end{eqnarray}
  
\noindent  Similar results to \eqref{ovpnj} and \eqref{esovp} hold also for $P_n(x_j)$ meaning that 
\begin{eqnarray}\label{ovepas}
\biggr( \delta_n(P(x_j)), \ j\in [1,r]\biggr)\stackrel{\mathcal{D}}{\rightsquigarrow }Z(\textbf{P})\stackrel{d }{\sim}\mathcal{N}(0,\Sigma_{ \textbf{P}}),\ \ \text{as}\ \ n\rightarrow+\infty,
\end{eqnarray}where $Z( \textbf{P})=(Z_{P(x_j)},j\in [1,r])^t
$ is a centered Gaussian random vector of dimension $r$ having respectively the following elements :
\begin{eqnarray}
&&\left(\Sigma_{\textbf{P}}\right)_{(i,j)}=(1-P(x_j)) \delta_{ij}-\sqrt{P(x_i)P(x_j)}  (1-\delta_{ij}), \ \ (i,j) \in [1,r]^2,
\end{eqnarray} and finally
 \begin{eqnarray*}
 \mathbb{E}\left(P_n(x_j)\right)=P(x_j),\ \ \ \mathbb{V}(P_n(x_j))=\frac{P(x_j)(1-P(x_j))}{n}.
 \end{eqnarray*}

 \bigskip \noindent 
 As a consequence, given $x_j\in X(\Omega)$, 
 we estimate 
discrete residual/past entropies and discrete residual/past inaccuracy measures at time $x_j$ from the sample $X_1,\cdots,X_n$ by their plug-in counterparts,
meaning that we replace the unknown  \textit{p.m.f.}, $p_{k}$ with its empirical estimate $ \widehat{p}_{n}^{(k)}$, computed from \eqref{pn}, in their expressions,  \textit{viz} :
\begin{eqnarray*}
 \mathcal{R}_X^{(n)}(x_j)&=&-\sum_{k=j}^{r} \frac{\widehat{p}_{n}^{(k)}}{\overline{P}_n(x_j)}\log \frac{\widehat{p}_{n}^{(k)}}{\overline{P}_n(x_j)},
\ \ \  
\mathcal{E}_{P}^{(n)}(X,x_j)=-\sum_{k=1}^j \frac{\widehat{p}_{n}^{(k)}}{P_n(x_j)}\log \left(\frac{\widehat{p}_{n}^{(k)}}{P_n(x_j)}
\right)
\\
\mathcal{R}_{(X,Y)}^{(n)}(x_j)&=&-\sum_{k=j}^{r}  \frac{\widehat{p}_{n}^{(k)}}{\overline{P}_n(x_j)}\log  \frac{q_k}{\overline{Q}(x_j)},
\ \ \text{and}\ \ \mathcal{P}_{(X,Y)}^{(n)}(x_j)=-\sum_{k=1}^j \frac{\widehat{p}_{n}^{(k)}}{P_n(x_j)}\log \frac{q_k}{Q(j)}.
\end{eqnarray*} 

\noindent Likewise, estimators of the  discrete cumulative residual/past entropies, of the discrete (mean) residual/past time, and of the cumulative residual/past inaccuracy measures are 
\begin{eqnarray*}
\mathcal{CR}_{X}^{(n)}&=&-\sum_{j=1}^{r-1}\overline{P}_n(x_j)\log \overline{P}_n(x_j),\ \ \
 \mathcal{CP}_X^{(n)}=-\sum_{j=1}^{r}P_n(x_j)\log P_n(x_j),\\
\mu_R^{(n)}(x_j)&=& \frac{1}{\overline{P}_n(x_j)}\sum_{k=j}^{r-1}\overline{P}_n(x_k),\ \ \mu_P^{(n)}(x_j)= \frac{1}{P_n(x_j)}\sum_{k=1}^jP_n(x_k),\\
 \mathcal{CR}_{(X,Y)}^{(n)}&=&-\sum_{j=1}^{r-1} \overline{P}_n(x_j)\log \overline{Q}(x_j),\ \ \ \text{and}\ \ \ \mathcal{CP}_{(X,Y)}^{(n)}=-\sum_{j=1}^{r} P_n(x_j)\log Q(x_j).
\end{eqnarray*}

\noindent In the following, we present asymptotic limits of these empirical estimators.
\section{Main contribution}\label{main-res} In this section, we look into the almost sure (\textit{a.s.}) convergence and asymptotic normality of the estimators defined in the previous section.
\subsection{Asymptotic behavior of the discrete residual/past entropies estimators at time $x_j$}$\,$\\
\noindent In the following proposition, we prove the almost sure convergence and the asymptotic normality of the estimator $ \mathcal{R}_X^{(n)}(x_j),\ \ j\in [1,r-1]$.\\

 \noindent 
 For a fixed $j\in [1,r-1]$, let
 \begin{eqnarray} 
 \label{reest}
 \mathcal{R}_X^{(n)}(x_j)&=&-\sum_{k=j}^{r-1} \frac{\widehat{p}_{n}^{(k)}}{\overline{P}_n(x_j)}\log \frac{\widehat{p}_{n}^{(k)}}{\overline{P}_n(x_j)}.
 \end{eqnarray}
   Denote\begin{eqnarray*}
   A_{R}(x_j)&=&\frac{ 1}{\overline{P} (j)} \sum_{k=j}^{r-1}\left\vert 1+\log \frac{p_k}{\overline{P}(x_j)}\right\vert \\
 \text{and}\ \ \sigma_{R}^2(x_j)&=& \sigma_{R,1}^{2}(j)+\sigma_{R,2}^{2}(j)+2\,\text{Cov}\left(T_{R,1}(j),T_{R,2}(j)\right)
\end{eqnarray*}
 where $T_{R,1}(j)\stackrel{d }{\sim}\mathcal{N}\left(0, \sigma_{R,1}^{2}(j)\right) $ and $T_{R,2}(j)\stackrel{d }{\sim}\mathcal{N}\left(0, \sigma_{R,2}^{2}(j)\right) $ with
 \begin{eqnarray*}
\sigma_{R,1}^{2}(j)&=&\frac{1}{(\overline{P}(x_j))^2 }\biggr(
   \sum_{k=j}^{r-1}p_k(1-p_k)\biggr[1+\log \left( \frac{p_k}{\overline{P}(x_j) }\right)\biggr]^2\\
  & & -  2 \sum_{\begin{tabular}{ c}
$(k,k')\in [j,r]^2$\\
$k \neq k'$\end{tabular}}p_kp_{k'}\biggr[1+\log \left( \frac{p_k}{\overline{P}(x_j) }\right)\biggr]\biggr[1+\log \left( \frac{p_{k'}}{\overline{P}(x_j) }\right)\biggr]\biggr)\\
  \text{and}\ \ \sigma_{R,2}^2(j)&=&\frac{1-\overline{P}(x_j)}{( \overline{P}(x_j))^3}\biggr[\sum_{k=j}^{r-1} p_k\left( 1+\log\frac{p_k}{\overline{P}(x_j) }\right)\biggr]^2.
 \end{eqnarray*}
 
\begin{proposition}
 \label{resient} For a fixed $j\in [1,r-1]$, let $
 \mathcal{R}_X^{(n)}(x_j)$ defined by \eqref{reest}. Then the following asymptotic results  hold
\begin{eqnarray}
&& \label{theojas} \limsup_{n\rightarrow+\infty}\frac{ \left\vert \mathcal{R}_X^{(n)}(x_j)- \mathcal{R}_X(x_j)\right \vert }{ a_{R,n}(p) }\leq A_{R}(x_j),\ \ \text{a.s.},\\
 && \label{theojan}  \sqrt{n}( \mathcal{R}_X^{(n)}(x_j)- \mathcal{R}_X(x_j))\stackrel{\mathcal{D}}{ \rightsquigarrow} \mathcal{N}\left(0,\sigma_{R}^2(x_j)
 \right),\ \ \text{as}\ \ n\rightarrow+\infty,
  \end{eqnarray}
  where 
  \begin{equation}\label{arnp}
   a_{R,n}(p)=\sup_{k\in [j,r]}\left\vert \frac{\widehat{p}_{n}^{(k)}}{\overline{P}_n(x_j)}-\frac{p_k}{\overline{P}(x_j)}\right\vert.
     \end{equation}

\end{proposition}

\begin{proof}$\,$\\
\noindent Fix $j\in [1,r-1]$ and define   
 \begin{eqnarray}\label{Deltko}
  \Delta_{R,n}(p_k)&=&\frac{\widehat{p}_{n}^{(k)} }{\overline{P}_n(x_j)}-\frac{p_k}{\overline{P}(x_j)},\ \ 
\forall\, k\in [j,r].
\end{eqnarray}  

\noindent Define the function $\psi \, :\, (0,+\infty)\rightarrow \,\mathbb{R}$ by $\psi(x)=x\log x$. \\
 \noindent For a fixed $k\in [j,r]$, we have
 \begin{eqnarray}\label{mvt1}
\notag \psi \left( \frac{\widehat{p}_{n}^{(k)} }{\overline{P}_n(x_j)} \right)&=&\psi \left( \frac{p_k}{\overline{P}(x_j) }+\Delta_{R,n}(p_k)\right)\\
&=&\psi \left(  \frac{p_k}{\overline{P}(x_j) }\right) +\Delta_{R,n}(p_k)\psi' \left( \frac{p_k}{\overline{P}(x_j) }+\theta_{1}(j,k)\Delta_{R,n}(p_k)\right),
\end{eqnarray} by applying the mean values theorem and where $\theta_1(j,k)$ is some number lying in $(0,1)$.\\
\noindent By applying again the mean values theorem to the derivative function $\psi'$ of $\psi$, we obtain
\begin{eqnarray*}
\psi' \left( \frac{p_k}{\overline{P}(x_j) }+\theta_{1}(j,k)\Delta_{R,n}(p_k)\right)&=&\psi' \left( \frac{p_k}{\overline{P}(x_j) }\right)\\
&&+ \theta_{1}(j,k)\Delta_{R,n}(p_k)\psi" \left( \frac{p_k}{\overline{P}(x_j) }+\theta_{2}(j,k)\Delta_{R,n}(p_k)\right),
\end{eqnarray*}where $\theta_{2}(j,k)\in (0,1)$. We can write \eqref{mvt1} as 
\begin{eqnarray*}
\notag \psi \left( \frac{\widehat{p}_{n}^{(k)} }{\overline{P}_n(x_j) }\right)&=&\psi \left(  \frac{p_k}{\overline{P}(x_j) }\right) +\Delta_{R,n}(p_k)\psi' \left( \frac{p_k}{\overline{P}(x_j) }\right)\\
&&\ \ \ \ \ \ \ \ \ \ +\ \ \ \ \theta_{1}(j,k)\left( \Delta_{R,n}(p_k)\right)^2 \psi" \left( \frac{p_k}{\overline{P}(x_j) }+\theta_{2}(j,k)\Delta_{R,n}(p_k)\right).
\end{eqnarray*}\noindent Now we have, by summation over $k\in [j,r]$
\begin{eqnarray}\label{mtvas}
\notag \mathcal{R}_X(x_j)- \mathcal{R}_X^{(n)}(x_j)&=&\sum_{k=j}^{r}\Delta_{R,n}(p_k)\psi' \left( \frac{p_k}{\overline{P}(x_j) }\right)\\
 &&\ \ \ \ + \ \ \ \sum_{k=j}^{r}\theta_{1}(j,k)\left( \Delta_{R,n}(p_k)\right)^2 \psi" \left( \frac{p_k}{\overline{P}(x_j) }+\theta_{2}(j,k)\Delta_{R,n}(p_k)\right),
\end{eqnarray}hence\begin{eqnarray*}
\left\vert \mathcal{R}_X(x_j)- \mathcal{R}_X^{(n)}(x_j)\right\vert  &\leq & a_{R,n}(p)\sum_{k=j}^{r}\left\vert \psi' \left( \frac{p_k}{\overline{P}(x_j) }\right)\right\vert\\
 &&\ \ \ \ \ \ \ \ \ \ \ + \ \ \  (a_{R,n}(p))^2\sum_{k=j}^{r}\left\vert \psi" \left( \frac{p_k}{\overline{P}(x_j) }+\theta_{2}(j,k)\Delta_{R,n}(p_k)\right)\right\vert.
\end{eqnarray*}

\noindent For $j\in [1,r-1]$ and $k\in [j,r]$, $ \Delta_{R,n}(p_k)$ can be re-expressed as  
\begin{eqnarray}\label{deltarn}
 \Delta_{R,n}(p_k)=\frac{1}{\overline{P}(x_j)}(\widehat{p}_{n}^{(k)}-p_k)- \widehat{p}_{n}^{(k)} \frac{(\overline{P}_n(x_j))- \overline{P}(x_j)}{ \overline{P}(x_j)\overline{P}_n(x_j)  }.
\end{eqnarray} Then using \eqref{maxanp} and for $n$ large enough, we have for any $k\in[j,r]$,
\begin{eqnarray*}
&&-a_n(p)\leq \widehat{p}_{n}^{(k)}-p_k\leq a_n(p)\ \ \text{and}\ \ -a_n(\overline{P})\leq \overline{P}_n(x_j)-\overline{P}(x_j)\leq a_n(\overline{P}).
\end{eqnarray*} Hence \begin{eqnarray}
\label{inegam} \sup_{k\in [j,r]}\left\vert \Delta_{R,n}(p_k)\right\vert &\leq & \frac{1}{\overline{P}(x_j)}|\widehat{p}_{n}^{(k)}-p_k|+\widehat{p}_{n}^{(k)} \left\vert \frac{ \overline{P}_n(x_j)-\overline{P}(x_j)}{ \overline{P}(x_j)\overline{P}_n(x_j)  }\right\vert\\
&\leq & \frac{a_n(p)}{\overline{P}(x_j)}+\frac{(a_n(p)+p_k)  a_n( \overline{P})}{\overline{P}(x_j)(\overline{P}(x_j)-a_n(\overline{P})) } ,
\notag \end{eqnarray}
which entails that $a_{R,n}(p)\stackrel{a.s.}{\longrightarrow}0$ as $n\rightarrow +\infty$ since $\max(a_n(p),a_n(\overline{P}))\stackrel{a.s.}{\longrightarrow}0$ as $n\rightarrow+\infty$.\\

\noindent Therefore, for a fixed $j\in [1,r-1]$, 
\begin{eqnarray*}
\limsup_{n\rightarrow +\infty}\frac{\left\vert \mathcal{R}_X(x_j)- \mathcal{R}_X^{(n)}(x_j)\right\vert}{a_{R,n}(p)} &\leq &\sum_{k=j}^{r}\left\vert \psi' \left( \frac{p_k}{\overline{P}(x_j) }\right)\right\vert=A_{R}(x_j)\ \ \text{a.s.}
\end{eqnarray*}since, as $n\rightarrow +\infty$ $$\sum_{k=j}^{r} \psi" \left( \frac{p_k}{\overline{P}(x_j) }+\theta_{2}(j,k)\Delta_{R,n}(p_k)\right)\rightarrow \sum_{k=j}^{r} \psi" \left( \frac{p_k}{\overline{P}(x_j) }\right)<\infty.$$
\noindent This proves the claim \eqref{theojas}.\\

\noindent Let prove the claim \eqref{theojan}. \\ 
\noindent Going back to \eqref{mtvas}, we have, for a fixed $j\in [1,r-1]$,

\begin{eqnarray*}
\sqrt{n}\left(\mathcal{R}_X(x_j)- \mathcal{R}_X^{(n)}(x_j) \right)=\sum_{k=j}^{r}  \sqrt{n}  \Delta_{R,n}(p_k) \psi' \left( \frac{p_k}{\overline{P}(x_j) }\right)+\sqrt{n}R_{n}(j),
\label{asumdel} 
\end{eqnarray*}
\noindent where $$R_{n}(j)=\sum_{k=j}^{r}\left( \Delta_{R,n}(p_k)\right)^2\theta_{1}(j,k) \psi" \left( \frac{p_k}{\overline{P}(x_j) }+\theta_{2}(j,k)\Delta_{R,n}(p_k)\right).$$ Asymptotically, from \eqref{deltarn} and using \eqref{maxanp}, we have for a fixed $j\in [1,r-1]$ and $k\in [j,r]$, 
\begin{eqnarray*} \Delta_{R,n}(p_k) \psi' \left( \frac{p_k}{\overline{P}(x_j) }\right) &\approx &  \left((\widehat{p}_{n}^{(k)}-p_k)\frac{1}{\overline{P}(x_j)}- (\overline{P}_n(x_j)- \overline{P}(x_j) \frac{p_k}{ (\overline{P}(x_j))^2  }\right) \psi' \left( \frac{p_k}{\overline{P}(x_j) }\right).
 \end{eqnarray*} 

First, it follows from \eqref{pnj}, that for $j\in [1,r-1]$, as $n\rightarrow+\infty$
\begin{eqnarray*}
&&\sum_{k=j}^{r} \sqrt{n}(\widehat{p}_{n}^{(k)}-p_k)\frac{1}{\overline{P}(x_j)}\psi' \left( \frac{p_k}{\overline{P}(x_j) }\right)\\
&&\ \ \ = \ \ \ \ \frac{1}{\overline{P}(x_j)}\sum_{k=j}^{r}\psi' \left( \frac{p_k}{\overline{P}(x_j) }\right)\sqrt{p_k}\delta_n(p_k) \\
&&\ \ \ \ \ \ \ \stackrel{\mathcal{D} }{\rightsquigarrow}  \frac{1}{\overline{P}(x_j)}\sum_{k=j}^{r}\psi' \left( \frac{p_k}{\overline{P}(x_j) }\right)\sqrt{p_k}Z_{p_k}=T_{R,1}(j),\ \ 
\end{eqnarray*} 
 where $\displaystyle T_{R,1}(j)\stackrel{d }{\sim }\mathcal{N}\left(0,\sigma_{R,1}^{2}(j)\right)$ with 
 \begin{eqnarray*}
\sigma_{R,1}^{2}(j)&=&\frac{1}{(\overline{P}(x_j))^2 }\biggr(
   \sum_{k=j}^{r}p_k(1-p_k)\biggr[1+\log \left( \frac{p_k}{\overline{P}(x_j) }\right)\biggr]^2\\
  & &\ \ -\ \ \  2\sum_{\begin{tabular}{ c}
$(k,k')\in [j,r-1]^2$\\
$k \neq k'$\end{tabular}}p_kp_{k'}\biggr[1+\log \left( \frac{p_k}{\overline{P}(x_j) }\right)\biggr]\biggr[1+\log \left( \frac{p_{k'}}{\overline{P}(x_j) }\right)\biggr]\biggr),
 \end{eqnarray*}since
 \begin{eqnarray*}
&& \mathbb{V}\left(  \frac{1}{\overline{P}(x_j)}\sum_{k=j}^{r}\psi' \left( \frac{p_k}{\overline{P}(x_j) }\right)\sqrt{p_k}Z_{p_k}\right)\\
&&=  \frac{1}{(\overline{P}(x_j))^2}\biggr[ \sum_{k=j}^{r}\mathbb{V}\left( \psi' \left( \frac{p_k}{\overline{P}(x_j) }\right)\sqrt{p_k}Z_{p_k}\right)\\
 &&+2\sum_{\begin{tabular}{ c}
$(k,k')\in [j,r-1]^2$\\
$k \neq k'$\end{tabular}}\mathbb{C}\text{ov} \biggr( \psi' \left( \frac{p_k}{\overline{P}(x_j) }\right)\sqrt{p_k}Z_{p_k}, \psi' \left( \frac{p_{k'}}{\overline{P}(x_j) }\right)\sqrt{p_{k'}}Z_{p_{k'}}\biggr)\biggr]\\
&&=  \frac{1}{(\overline{P}(x_j))^2}\biggr[ \sum_{k=j}^{r}p_k \left( \psi' \left( \frac{p_k}{\overline{P}(x_j) }\right)\right)^2\mathbb{V}\left(Z_{p_k}\right)\\
 &&+2\sum_{\begin{tabular}{ c}
$(k,k')\in [j,r-1]^2$\\
$k \neq k'$\end{tabular}}\sqrt{p}_k\sqrt{p}_{k'}\psi' \left( \frac{p_k}{\overline{P}(x_j) }\right) \psi' \left( \frac{p_{k'}}{\overline{P}(x_j) }\right)\mathbb{C}\text{ov} \biggr( Z_{p_k} ,Z_{p_{k'}}\biggr)\biggr]
\\
&&=\frac{1}{(\overline{P}(x_j))^2 }\biggr(
   \sum_{k=j}^{r}p_k(1-p_k)\biggr[1+\log \left( \frac{p_k}{\overline{P}(x_j) }\right)\biggr]^2\\
  & &- 2\sum_{\begin{tabular}{ c}
$(k,k')\in [j,r-1]^2$\\
$k \neq k'$\end{tabular}}p_kp_{k'}\biggr[1+\log \left( \frac{p_k}{\overline{P}(x_j) }\right)\biggr]\biggr[1+\log \left( \frac{p_{k'}}{\overline{P}(x_j) }\right)\biggr]\biggr)
 \end{eqnarray*}
\noindent Second, 
 it holds from \eqref{norovp} that for $j\in [1,r-1]$,
   \begin{eqnarray*}
 && \sum_{k=j}^{r} \sqrt{n}(\overline{P}_n(x_j)-\overline{P}(x_j))\frac{p_k}{\left( \overline{P}(x_j)\right)^2 }\psi' \left( \frac{p_k}{\overline{P}(x_j) }\right)\\
 &&\ \ \ \ \ \ \ = \ \ \  \biggr[ \sum_{k=j}^{r} \frac{p_k}{\left( \overline{P}(x_j)\right)^2 }\psi' \left( \frac{p_k}{\overline{P}(x_j) }\right)\biggr] \sqrt{\overline{P}(x_j)} \delta_n(\overline{P}(x_j))\\
&&
\ \ \ \ \ \ \ \ \ \  \stackrel{\mathcal{D} }{\rightsquigarrow} \ \ \ \ \ \biggr[ \sum_{k=j}^{r} \frac{p_k}{\left( \overline{P}(x_j)\right)^2 }\psi' \left( \frac{p_k}{\overline{P}(x_j) }\right)\biggr]Z_{\overline{P}(x_j)}=T_{R,2}(j)% \mathcal{N}\left(0,\sigma_2^2(j)\right)
,\end{eqnarray*}as $n\rightarrow+\infty$ and 
 where $Z_{\overline{P}(x_j)}\stackrel{d}{\sim} N\left(0,\overline{P}(x_j)(1-\overline{P}(x_j)\right)$. \\
 \noindent Therefore $\displaystyle T_{R,2}(j)\stackrel{d }{\sim }\mathcal{N}\left(0,\sigma_{R,2}^{2}(j)\right)$ with 
  \begin{eqnarray*}
\sigma_{R,2}^2(j)&=&\frac{1-\overline{P}(x_j)}{( \overline{P}(x_j))^3}\biggr[\sum_{k=j}^{r} p_k\left( 1+\log\frac{p_k}{\overline{P}(x_j) }\right)\biggr]^2. \end{eqnarray*}
Thus \begin{eqnarray}
\sum_{k=j}^{r}  \sqrt{n}\Delta_{R,n}(p_k)\psi' \left( \frac{p_k}{\overline{P}(x_j) }\right)
\stackrel{\mathcal{D} }{\rightsquigarrow}\mathcal{N}\left(0,\sigma_{R}^2(j)\right),
\end{eqnarray}with \begin{equation*}
\sigma_{R}^2(j)=\sigma_{R,1}^2(j)+\sigma_{R,2}^2(j)+2\,\text{Cov}\left(T_{R,1}(j),T_{R,2}(j)\right),\ \ j\in [1,r-1].
\end{equation*} 
 It remains to prove that $\sqrt{n} R_{n}(j)$ converges in probability to $0$, as $n$ tends to infinity. We have
for $j\in [1,r-1]$ 
 \begin{equation} \label{r1n}
\left\vert \sqrt{n}R_{n}(j)\right\vert \leq \sqrt{n}\left(a_{R,n}(p)\right)^{2}\sum_{k=j}^{r}\left\vert \psi" \left( \frac{p_k}{\overline{P}(x_j) }+\theta_{2}(j,k)\Delta_{R,n}(p_k)\right)\right\vert.
\end{equation}

\noindent Let show that $$\sqrt{n}\left(a_{R,n}(p)\right)^{2}=o_{\mathbb{P}}(1).$$ 

\noindent For $n$ large enough and for any $k\in[j,r]$, we have from \eqref{inegam}, the following inequality  
 \begin{eqnarray*}
\left\vert \Delta_{R,n}(p_k)\right \vert &\leq &   \lambda_{n}(p_k) +\lambda_{n}(\overline{P}(x_j)),
\end{eqnarray*}where \begin{eqnarray*}
\lambda_{n}(p_k)=\frac{1}{\overline{P}(x_j)}|\widehat{p}_{n}^{(k)}-p_k|\  
\ \ \text{and}\ \ \ \lambda_{n}(\overline{P}(x_j))= (a_n(p)+p_k) \left\vert \frac{ \overline{P}_n(x_j)-\overline{P}(x_j)}{(\overline{P}(x_j)-a_n(\overline{P})) \overline{P}(x_j)}\right\vert.
\end{eqnarray*}

\noindent By the Bienaym\'e-Tchebychev inequality, we have, for any $\epsilon >0$ and for $k\in[j,r]$,

\begin{eqnarray}\label{pal}
\notag \mathbb{P}\left( |\lambda_{n}(p_k)|\geq \frac{\sqrt{\varepsilon}}{2 n^{1/4}}\right)&=& \mathbb{P}\left( |\widehat{p}_{n}^{(k)}-p_k|\geq \frac{\sqrt{\varepsilon}\times \overline{P}(x_j)}{2 n^{1/4}}\right)\\
 &\leq & 
 \frac{4p_k(1-p_k)}{\varepsilon n^{1/2}\times (\overline{P}(x_j))^2}\\
\notag \text{and}\ \ \mathbb{P}\left(|\lambda_{n}(\overline{P}(x_j))|\geq \frac{\sqrt{\varepsilon}}{2n^{1/4}}\right)&=&
\mathbb{P}\left(|\overline{P}_n(x_j)-\overline{P}(x_j)|\geq \frac{\sqrt{\varepsilon}\times ((\overline{P}(x_j)- a_n(\overline{P}))\overline{P}(x_j))}{2n^{1/4}(a_n(p)+p_k)}\right)\\
\label{pbet} &\leq &\frac{4\overline{P}(x_j)(1-\overline{P}(x_j))(a_n(p)+p_k)^2}{\varepsilon n^{1/2}\times (\overline{P}(x_j)- a_n(\overline{P}))\overline{P}(x_j))^2}.
\end{eqnarray}
\noindent combining \eqref{pal} and \eqref{pbet},  
\begin{eqnarray*}
\mathbb{P}\left( |\lambda_{n}(p_k)|\geq \frac{\varepsilon}{2 n^{1/4}}\right)+\mathbb{P}\left(|\lambda_{n}(\overline{P}(x_j))|\geq \frac{\varepsilon}{2n^{1/4}}\right) &\leq &\frac{4}{ \varepsilon n^{1/2}\times (\overline{P}(x_j))^2   }\biggr( p_k(1-p_k)\\
&& + \frac{\overline{P}(x_j)(1-\overline{P}(x_j))(a_n(p)+p_k)^2}{(\overline{P}(x_j)-a_n(\overline{P}))^2}\biggr).
\end{eqnarray*}
\noindent Hence, for any $k\in [j,r]$, \begin{eqnarray*}
\mathbb{P}\left(\sqrt{n} \left\vert  \Delta_{R,n}(p_k)\right\vert^2\geq \varepsilon \right)&=&\mathbb{P}\left( \left\vert  \Delta_{R,n}(p_k)\right\vert\geq \frac{\sqrt{\varepsilon}}{2 n^{1/4}}+\frac{\sqrt{\varepsilon}}{2 n^{1/4}} \right)\\
&\leq & \mathbb{P}\left( |\lambda_{n}(p_k)| \geq \frac{\sqrt{\varepsilon}}{2 n^{1/4}}\right) + \mathbb{P}\left(|\lambda_{n}(\overline{P}(x_j))|\geq \frac{\sqrt{\varepsilon}}{2 n^{1/4}}\right) \\
&\leq &\frac{4}{ \varepsilon n^{1/2}\times (\overline{P}(x_j))^2   }\biggr( p_k(1-p_k)\\
&& \ \ \ \ \ \ \ +\ \ \ \frac{\overline{P}(x_j)(1-\overline{P}(x_j))(a_n(p)+p_k)^2}{(\overline{P}(x_j)-a_n(\overline{P}))^2}\biggr),\end{eqnarray*}which implies that $\sqrt{n}\left(a_{R,n}(p)\right)^{2}$ converges in probability to $0$ as $n\rightarrow+\infty$
.\\

\noindent Therefore, from \eqref{r1n},  $\sqrt{n}R_{1,n}=0_{\mathbb{P}}(1)$. Accordingly, for any fixed $j\in [1,r-1]$, we have
$$
\sqrt{n}( \mathcal{R}_X^{(n)}(x_j)- \mathcal{R}_X(x_j))\stackrel{\mathcal{D}}{ \rightsquigarrow} \mathcal{N}\left(0,\sigma_{R}^2(x_j)
 \right),\ \ \text{as}\ \ n\rightarrow+\infty.$$

\noindent This proves the claim \eqref{theojan} and ends the proof of the Proposition \ref{resient}.
\end{proof}

\bigskip \noindent The Proposition \ref{pastentr} below establishes the asymptotic behavior of the discrete past entropy estimator at time $x_j,\ \  j\in J$. \\

\noindent 
We omit the proof which is almost exactly the same as that of Proposition \ref{resient}.\\

\bigskip \noindent 
 For a fixed $j\in [1,r]$,  let \begin{equation}\label{paest}
 \mathcal{P}_X^{(n)}(x_j)=-\sum_{k=1}^{j} \frac{\widehat{p}_{n}^{(k)}}{P_n(x_j)}\log \frac{\widehat{p}_{n}^{(k)}}{P_n(x_j)}
\end{equation} and denote
\begin{eqnarray*}
 A_{P}(x_j)&=&\frac{ 1}{P(x_j)} \sum_{k=1}^{j}\left\vert 1+\log \frac{p_k}{P(x_j)}\right\vert \\
 \sigma_{P}^2(x_j)&=& \sigma_{P,1}^{2}(j)+\sigma_{P,2}^{2}(j)+2\,\text{Cov}\left(T_{P,1}(j),T_{P,2}(j)\right)
\end{eqnarray*}
 where 
 \begin{eqnarray*}
\sigma_{P,1}^{2}(j)&=&\frac{1}{(P(x_j))^2 }\biggr(
   \sum_{k=1}^{j}p_k(1-p_k)\biggr[1+\log \left( \frac{p_k}{P(x_j) }\right)\biggr]^2\\
  & &\ \ \ \ \ \ \ \  -\ \ \  2\ \ \sum_{%(k,k')\in [1,r-1]^2,
  k\neq k'
  }p_kp_{k'}\biggr[1+\log \left( \frac{p_k}{P(x_j) }\right)\biggr]\biggr[1+\log \left( \frac{p_{k'}}{P(x_j) }\right)\biggr]\biggr)\\
  \sigma_{P,2}^2(j)&=&\frac{1- P(x_j)}{\left(P(x_j)\right)^{3} }\biggr[\sum_{k=1}^{j} p_k\left( 1+\log\frac{p_k}{P(x_j) }\right)\biggr]^2,
 \end{eqnarray*} $$T_{P,1}(j)\stackrel{d}{\sim} \mathcal{N}(0,\sigma_{P,1}^{2}(j))\ \ \text{and}\ \ T_{P,2}(j)\stackrel{d}{\sim} \mathcal{N}(0,\sigma_{P,2}^{2}(j)).$$
 
\begin{proposition}
 \label{pastentr} For a fixed $j\in [1,r]$, let $\mathcal{P}_X^{(n)}(x_j)$ defined by \eqref{paest}.  Then the following asymptotic results hold
\begin{eqnarray}
&& \label{theopast} \limsup_{n\rightarrow+\infty}\frac{ \left\vert \mathcal{P}_X^{(n)}(x_j)- \mathcal{P}_X(x_j)\right \vert }{ a_{P,n}(p) }\leq A_{P}(x_j),\ \ \text{a.s.},\\
 && \label{theopasn}  \sqrt{n}( \mathcal{P}_X^{(n)}(x_j)- \mathcal{P}_X(x_j))\stackrel{\mathcal{D}}{ \rightsquigarrow} \mathcal{N}\left(0,\sigma_{P}^2(x_j)
 \right),\ \ \text{as}\ \ n\rightarrow+\infty.
  \end{eqnarray}where 
  \begin{equation}\label{appn}
      a_{P,n}(p)=\sup_{k\in [1,j]}\left\vert \frac{\widehat{p}_{n}^{(k)}}{P_n(x_j)}-\frac{p_k}{P(x_j)}\right\vert.
\end{equation}
\end{proposition}
\subsection{Asymptotic behavior of the discrete cumulative residual/past entropies estimators}$\,$\\

\noindent We focus here on the almost sure convergence and the asymptotic normality of the cumulative residual inaccuracy measure estimator between $X$ and $Y$.\\

\bigskip \noindent Let \begin{equation}\label{creest}
\mathcal{CR}_{X}^{(n)}=-\sum_{j=1}^{r-1}\overline{P}_n(x_j)\log \overline{P}_n(x_j)
\end{equation}and denote
\begin{eqnarray*}
A_{CR}&=&\sum_{j=1}^{r-1}\left\vert 1+\log \overline{P}(x_j)\right\vert \\
\sigma_{CR}^2 &=&\sum_{j=1}^{r-1} \overline{P}(x_j)(1-\overline{P}(x_j))\left(1+\log \overline{P}(x_j)\right) ^2\\
 &&\ \ \ \ -\ \ 2\sum_{\begin{tabular}{ c}
$(j,j')\in [1,r-1]^2$\\
$j\neq j'$\end{tabular}}\overline{P}(x_j)\overline{P}(x_{j'})\left(1+\log \overline{P}(x_j)\right)\left(1+\log \overline{P}(x_{j'})\right).
\end{eqnarray*}

\begin{proposition}
 \label{cumres}Let $\mathcal{CR}_{X}^{(n)}$ defined by \eqref{creest}. Then  the following asymptotic results hold
 \begin{eqnarray}
\label{creas1} &&\limsup_{n\rightarrow+\infty}\frac{\left\vert \mathcal{CR}_{X}^{(n)}-\mathcal{CR}_{X}\right\vert }{a_n(P)}\leq A_{CR},\ \ \text{a.s.}\\
 &&\label{crean1} \sqrt{n}\left( \mathcal{CR}_{X}^{(n)}-\mathcal{CR}_{X}\right)\stackrel{\mathcal{D} }{\rightsquigarrow}\mathcal{N}(0,\sigma_{CR}^2),\ \ \text{as}\ \ n\rightarrow+\infty.
 \end{eqnarray}
\end{proposition}
\begin{proof} The proof is very similar to that of Proposition \ref{resient} since we use the same technics, in the circumstances the mean values theorem applied to $\psi'$ and to $\psi"$. So, for sake of shorten, we skeep some steps$\,$\\

\noindent Let $j\in [1,r-1]$, then 
 by the same approach as previously, we obtain
\begin{eqnarray}
\notag \mathcal{CR}_{X}-\mathcal{CR}_{X}^{(n)}&=&\sum_{j=1}^{r-1}\Delta_n(\overline{P}(x_j))\psi'\left(\overline{P}(x_j)\right)\\
\label{crenas} && \ \ \ \ +\ \ \ \ \sum_{j=1}^{r-1}\theta_{1}(j)\left(\Delta_n(\overline{P}(x_j))\right)^2\psi"\left( \overline{P}(x_j)+\theta_{2}(j)\Delta_n(\overline{P}(x_j))\right)
\end{eqnarray}where $(\theta_1(j),\theta_2(j))\in (0,1)^2$.
 Therefore, using \eqref{maxanp},
\begin{eqnarray*}
\limsup_{n\rightarrow+\infty}\frac{\left\vert \mathcal{CR}_{X}^{(n)}-\mathcal{CR}_{X}\right\vert }{a_n(\overline{P})}\leq \sum_{j=1}^{r-1}\left\vert \psi'\left(\overline{P}(x_j)\right)\right\vert=A_{CR},\ \ \text{a.s.}
\end{eqnarray*}since $$\sum_{j=1}^{r-1}\left\vert \psi"\left( \overline{P}(x_j)+\theta_{2}(j)\Delta_n(\overline{P}(x_j))\right)\right\vert \rightarrow \sum_{j=1}^{r-1}\left\vert \psi"\left( \overline{P}(x_j)\right)\right\vert<\infty,\ \ \text{as}\ \ n\rightarrow +\infty.$$
That confirms the claim \eqref{creas1}.
\\

\noindent Now, from \eqref{crenas}, we have 
\begin{eqnarray*}
\sqrt{n}\left(\mathcal{CR}_{X}-\mathcal{CR}_{X}^{(n)}\right)&=&\sum_{j=1}^{r-1}\sqrt{n}\left( \overline{P}_n(x_j)-\overline{P}(x_j)\right) \psi'\left(\overline{P}(x_j)\right)+ \sqrt{n}R_{n}   
\end{eqnarray*}where $$R_{n}= \sum_{j=1}^{r-1}\theta_{1}(j)\left(\Delta_n(\overline{P}(x_j))\right)^2\psi"\left( \overline{P}(x_j)+\theta_{2}(j)\Delta_n(\overline{P}(x_j))\right).
$$

\noindent Therefore, by \eqref{ovepas}, we obtain that
\begin{eqnarray*}
&& \sum_{j=1}^{r-1}\sqrt{n}\left( \overline{P}_n(x_j)-\overline{P}(x_j)\right)\psi'\left(\overline{P}(x_j)\right)\\
&&\ \ \ \ \ =\ \ \ \sum_{j=1}^{r-1}\psi'\left(\overline{P}(x_j)\right)\sqrt{\overline{P}(x_j)}\delta_n(\overline{P}(x_j))\\
&&\ \ \ \ \ \  \ \ \ \ \ \  
\stackrel{\mathcal{D} }{\rightsquigarrow} \sum_{j=1}^{r-1}\psi'\left(\overline{P}(x_j)\right)\sqrt{\overline{P}(x_j)} Z_{\overline{P}(x_j)},\ \ \text{as}\  \ n\rightarrow+\infty,
\end{eqnarray*} which follows a centered normal law with asymptotic  variance $\sigma_{CR}^2$ since 
\begin{eqnarray*}
&& \mathbb{V}\left(\sum_{j=1}^{r-1}\psi'\left(\overline{P}(x_j)\right)\sqrt{\overline{P}(x_j)} Z_{\overline{P}(x_j)} \right)\\
&&  =\sum_{j=1}^{r-1}\mathbb{V}\left(\psi'\left(\overline{P}(x_j)\right)\sqrt{\overline{P}(x_j)} Z_{\overline{P}(x_j)} \right)\\
&&+2\sum_{\begin{tabular}{ c}
$(j,j')\in [1,r-1]^2$\\
$j\neq j'$\end{tabular}}\mathbb{C}\text{ov} \biggr(   \psi'\left(\overline{P}(x_j)\right)\sqrt{\overline{P}(x_j)} Z_{\overline{P}(x_j)} ,\psi'\left(\overline{P}(x_{j'})\right)\sqrt{\overline{P}(x_{j'})} Z_{\overline{P}(x_{j'})}\biggr)\\
&=&\sum_{j=1}^{r-1}\overline{P}(x_j)\left( \psi'\left(\overline{P}(x_j)\right)\right)^2 \mathbb{V}\left(Z_{\overline{P}(x_j)} \right)\\
&&+2\sum_{\begin{tabular}{ c}
$(j,j')\in [1,r-1]^2$\\
$j\neq j'$\end{tabular}}\sqrt{\overline{P}(x_j)}\sqrt{\overline{P}(x_{j'})}\psi'\left(\overline{P}(x_j)\right)\psi'\left(\overline{P}(x_{j'})\right) \times \mathbb{C}\text{ov} \biggr(    Z_{\overline{P}(x_j)} , Z_{\overline{P}(x_{j'})}\biggr)\\
&=&\sum_{j=1}^{r-1}\overline{P}(x_j)(1-\overline{P}(x_j))\left( \psi'\left(\overline{P}(x_j)\right)\right)^2 \\
&&+2\sum_{\begin{tabular}{ c}
$(j,j')\in [1,r-1]^2$\\
$j\neq j'$\end{tabular}}\overline{P}(x_j)\overline{P}(x_{j'})\psi'\left(\overline{P}(x_j)\right)\psi'\left(\overline{P}(x_{j'})\right)\\
&=&\sum_{j=1}^{r-1} \overline{P}(x_j)(1-\overline{P}(x_j))\left(1+\log \overline{P}(x_j)\right) ^2\\
 &&-2\sum_{\begin{tabular}{ c}
$(j,j')\in [1,r-1]^2$\\
$j\neq j'$\end{tabular}}
\overline{P}(x_j)\overline{P}(x_{j'})\left(1+\log \overline{P}(x_j)\right)\left(1+\log \overline{P}(x_{j'})\right).
\end{eqnarray*}

\noindent Finally, the proof will be complete if we show that $\sqrt{n}R_{n}$ converges, in probability, to zero, as $n$ tends to infinity. We have
\begin{equation} \label{r1np}
\left\vert \sqrt{n}R_{n}\right\vert \leq \sqrt{n}\left(a_{n}(\overline{P})\right)^{2}\sum_{j=1}^{r-1}\left\vert \psi"\left( \overline{P}(x_j)+\theta_{2}(j)\Delta_n(\overline{P}(x_j))\right)\right\vert.
\end{equation}

\bigskip \noindent Let show that $$\sqrt{n}\left(a_{n}(\overline{P})\right)^{2}=o_{\mathbb{P}}(1).$$ By the Bienaym\'e-Tchebychev inequality, we have, for any $\epsilon >0$ and for $j\in[1,r-1]$,
\begin{eqnarray*}
\mathbb{P}(\sqrt{n}\left(\overline{P}_n(x_j)-\overline{P}(x_j)\right)^2\geq \epsilon)=\mathbb{P}\left(\left\vert \overline{P}_n(x_j)-\overline{P}(x_j)\right\vert\geq  \frac{\sqrt{\epsilon} }{n^{1/4}}\right)\leq \frac{\overline{P}_n(x_j)-\overline{P}(x_j)}{\epsilon n^{1/2}},
\end{eqnarray*}which implies that $\sqrt{n}\left(a_{n}(\overline{P})\right)^{2}$ converges in probability to $0$ as $n\rightarrow+\infty$.\\

\noindent Finally from \eqref{r1np} we have $\sqrt{n}
R_{n}\stackrel{\mathbb{P}}{\rightarrow} 0 \text{ as } n\rightarrow +\infty$ which implies 
$$\sqrt{n}\left( \mathcal{CR}_{X}^{(n)}-\mathcal{CR}_{X}\right)\stackrel{\mathcal{D} }{\rightsquigarrow}\mathcal{N}(0,\sigma_{CR}^2),\ \ \text{as}\ \ n\rightarrow+\infty.$$

\noindent This proves the claim \eqref{crean1}.\\

\noindent Hence the proof of the Proposition \ref{cumres} is complete.

\end{proof}

\bigskip \noindent 
In analogy with Proposition \ref{cumres}, the Proposition \ref{cumpent} below establishes the asymptotic behavior of the discrete cumulative past entropy estimator.  The proof
is omitted being similar
to that of Proposition \ref{cumres}.
\\

\noindent \noindent  
Let  \begin{equation}\label{cpeest}
\mathcal{CP}_X^{(n)}=-\sum_{j=1}^{r}P_n(x_j)\log P_n(x_j)
\end{equation} and denote
\begin{eqnarray*}
A_{CP}&=&\sum_{j=1}^{r}\left\vert 1+\log P(x_j)\right\vert \\
\sigma_{CP}^2 &=&\sum_{j=1}^{r} P(x_j)(1-P(x_j))\left(1+\log P(x_j)\right) ^2\\
 &&\ \ \ \ -\ \ 2\sum_{\begin{tabular}{ c}
$(j,j')\in [1,r]^2$\\
$j\neq j'$\end{tabular}}
P(x_j)P(x_{j'})\left(1+\log P(x_j)\right)\left(1+\log P(x_{j'})\right).
\end{eqnarray*}
\begin{proposition}\label{cumpent}
Let $\mathcal{CP}_X^{(n)}$ defined \eqref{cpeest}.  Then the following asymptotic results hold
\begin{eqnarray*}
&&\limsup_{n\rightarrow+\infty}\frac{\mathcal{CP}_X^{(n)}-\mathcal{CP}_X}{a_n(P)}\leq A_{CP},\ \ \text{a.s.}\\
&&\sqrt{n}(\mathcal{CP}_X^{(n)}-\mathcal{CP}_X)\stackrel{\mathcal{D} }{\rightsquigarrow}\mathcal{N}\left(0,\sigma_{CP}^2\right),\ \ \text{as}\ \ n\rightarrow+\infty.
\end{eqnarray*}
\end{proposition}
\subsection{Asymptotic behavior of discrete mean residual/past lifetime estimators}
$\,$\\

\noindent Given $j\in [1,r-1],$ we establish asymptotic limits for $\mu_R^{(n)}(x_j)$ and for $\mu_P^{(n)}(x_j)$.\\

\bigskip \noindent For a fixed $j\in [1,r-1],$ let \begin{equation}\label{meanr}
\mu_R^{(n)}(x_j)= \frac{1}{\overline{P}_n(x_j)}\sum_{k=j}^{r-1}\overline{P}_n(x_k),
\end{equation}
and denote \begin{eqnarray*}
A_{\mu_R}(x_j)&=&\frac{r-j}{\overline{P}(x_j)}+\frac{1}{(\overline{P}(x_j))^2}\sum_{k=j}^{r-1}\overline{P}(x_k)
\\
\sigma_{\mu_R}^2(j)&=&\sigma_{\mu_R,1}^2(j)+\sigma_{\mu_R,2}^2(j)+2\,\text{Cov}\left(T_{\mu_R,1}(j),T_{\mu_R,2}(j)\right),
\end{eqnarray*}where
\begin{eqnarray*}
T_{\mu_R,1}(j)\stackrel{d}{\sim}\mathcal{N}(0,\sigma_{\mu_R,1}^2(j))\ \ \text{and}\ \ T_{\mu_R,2}(j)\stackrel{d}{\sim}\mathcal{N}(0,\sigma_{\mu_R,2}^2(j))
\end{eqnarray*} with \begin{eqnarray*}
\sigma_{\mu_R,1}^2(j)&=&\frac{1}{(\overline{P}(x_j))^2}\biggr[ \sum_{k=j}^{r-1} \overline{P}(x_k)(1-\overline{P}(x_k))-2\sum_{\begin{tabular}{ c}
$(k,k')\in [j,r-1]^2$\\
$k\neq k'$\end{tabular}} \overline{P}(x_k) \overline{P}(x_{k'})\biggr]\\
\text{and} \ \ \sigma_{\mu_R,2}^2(j)&=&\frac{1-\overline{P}(x_j)}{( \overline{P}(x_j))^3}\biggr[\sum_{k=j}^{r-1} \overline{P}(x_k)\biggr]^2.
\end{eqnarray*}

\begin{proposition}\label{promur}For $j\in [1,r-1]$, let $
\mu_R^{(n)}(x_j)$ defined by \eqref{meanr}. Then the following asymptotic results hold
\begin{eqnarray}
\label{muras} \limsup_{n\rightarrow+\infty}\frac{|\mu_R^{(n)}(x_j)-\mu_R(x_j) }{a_n(\overline{P})}&\leq & A_{\mu_R}(x_j),\ \ \text{a.s.}\\
\sqrt{n}(\mu_R^{(n)}(x_j)-\mu_R(x_j))&\stackrel{\mathcal{D} }{\rightsquigarrow}& \mathcal{N}(0,
\sigma_{\mu_R}^2(x_j)),\ \ \text{as}\ \ n\rightarrow+\infty.\label{murasan}
\end{eqnarray}
\end{proposition}
\begin{proof}

\bigskip \noindent Fix $j\in [1,r-1]$, then we have    
 \begin{eqnarray}\label{mras}
\notag \mu_R^{(n)}(x_j)-\mu_R(x_j)&=& \frac{1}{\overline{P}_n(x_j)}\sum_{k=j}^{r-1}\overline{P}_n(x_k)-\frac{1}{\overline{P}(x_j)}\sum_{k=j}^{r-1}\overline{P}(x_k)\\
&=& \frac{1}{\overline{P}_n(x_j)} \sum_{k=j}^{r-1}(\overline{P}_n(x_k)-\overline{P}(x_k))- \frac{
\overline{P}_n(x_j)-\overline{P}(x_j)
}{\overline{P}_n(x_j)\overline{P}(x_j)}\sum_{k=j}^{r-1}\overline{P}(x_k).
\end{eqnarray} Which gives, for $n$ large enough
\begin{eqnarray*}
\left\vert \mu_R^{(n)}(x_j)-\mu_R(x_j)\right\vert &\leq &
\frac{(r-j)a_n(\overline{P})}{\overline{P}(x_j)-a_n(\overline{P}) }+\frac{a_n(\overline{P}) }{(\overline{P}(x_j)-a_n(\overline{P}))\overline{P}(x_j)}\sum_{k=j}^{r-1}\overline{P}(x_k).
\end{eqnarray*}Therefore
\begin{eqnarray*}
\limsup_{n\rightarrow+\infty}\frac{\left\vert \mu_R^{(n)}(x_j)-\mu_R(x_j)\right\vert}{a_n(\overline{P})}&\leq &
\frac{r-j}{\overline{P}(x_j)}+\frac{1}{(\overline{P}(x_j))^2}\sum_{k=j}^{r-1}\overline{P}(x_k),\ \ \text{a.s},
\end{eqnarray*}which proves the claim \eqref{muras}. \\

\noindent Let prove the claim \eqref{murasan}. \\
\noindent Going back to \eqref{mras} and  using \eqref{maxanp}, we have asymptotically, for a fixed $j\in [1,r-1]$ 

\begin{eqnarray*}
\sqrt(\mu_R^{(n)}(x_j)-\mu_R(x_j))&\approx &T_{\mu_R,n}^{(1)}-T_{\mu_R,n}^{(2)}
\end{eqnarray*} where 
\begin{eqnarray*}
\displaystyle T_{\mu_R,n}^{(1)}(j)&=& \frac{1}{\overline{P}(x_j)} \sum_{k=j}^{r-1}\sqrt{n}(\overline{P}_n(x_k)-\overline{P}(x_k))\\ 
\text{and}\ \ 
T_{\mu_R,n}^{(2)}&=&\frac{
\sqrt{n}(\overline{P}_n(x_j)-\overline{P}(x_j))
}{(\overline{P}(x_j))^2}\sum_{k=j}^{r-1}\overline{P}(x_k).
\end{eqnarray*}

\noindent We already know that, for $j\in [1,r-1]$, as $n \rightarrow+\infty,$ 
\begin{eqnarray*}
T_{\mu_R,n}^{(1)}(j)&\stackrel{\mathcal{D} }{\rightsquigarrow}&
T_{\mu_R,1}(j)\stackrel{d}{\sim}\mathcal{N}(0,\sigma_{\mu_R,1}^2(j))\\
 \ \text{and}\ \ T_{\mu_R,n}^{(1)}(j)&\stackrel{\mathcal{D} }{\rightsquigarrow}& T_{\mu_R,2}(j)\stackrel{d}{\sim}\mathcal{N}(0,\sigma_{\mu_R,2}^2(j)),
\end{eqnarray*}with 
\begin{eqnarray*}
\sigma_{\mu_R,1}^2(j)&=&\frac{1}{(\overline{P}(x_j))^2}\biggr[ \sum_{k=j}^{r-1} \overline{P}(x_k)(1-\overline{P}(x_k)) -2\sum_{\begin{tabular}{ c}
$(k,k')\in [j,r-1]^2$\\
$k\neq k'$\end{tabular}} \overline{P}(x_k) \overline{P}(x_{k'})\biggr]\\
\text{and} \ \ \sigma_{\mu_R,2}^2(j)&=&\frac{1-\overline{P}(x_j)}{( \overline{P}(x_j))^3}\biggr[\sum_{k=j}^{r-1} \overline{P}(x_k)\biggr]^2.
\end{eqnarray*}
Therefore 
\begin{eqnarray*}
\sqrt{n}(\mu_R^{(n)}(x_j)-\mu_R(x_j))\stackrel{\mathcal{D} }{\rightsquigarrow}\mathcal{N}(0,\sigma_{\mu_R}^2(j)).
\end{eqnarray*}with \begin{equation*}
\sigma_{\mu_R}^2(j)=\sigma_{\mu_R,1}^2(j)+\sigma_{\mu_R,2}^2(j)+2\,\text{Cov}\left(T_{\mu_R,1}(j),T_{\mu_R,2}(j)\right),\ \ j\in [1,r-1].
\end{equation*}
Which confirms the claim \eqref{murasan} and ends the proof of the proposition. 

\end{proof}

\bigskip \noindent The Proposition \ref{propast} below establishes the asymptotic behavior of the discrete inactivity lifetime estimator at time $x_j,\ \  j\in J$. \\

\noindent 
We omit the proof which is almost exactly the same as that of Proposition \ref{promur}.\\

\bigskip \noindent For a fixed $j\in [1,r-1],$ let \begin{equation}\label{meanp}
 \mu_P^{(n)}(x_j)= \frac{1}{P_n(x_j)}\sum_{k=1}^jP_n(x_k),
\end{equation}and denote \begin{eqnarray*}
A_{\mu_P}(x_j)&=&\frac{j}{P(x_j)}+\frac{1}{(P(x_j))^2}\sum_{k=1}^{j}P(x_k)
\\
\sigma_{\mu_P}^2(j)&=&\sigma_{\mu_P,1}^2(j)+\sigma_{\mu_P,2}^2(j)+2\,\text{Cov}\left(T_{\mu_P,1}(j),T_{\mu_P,2}(j)\right),
\end{eqnarray*}where
\begin{eqnarray*}
T_{\mu_P,1}(j)\stackrel{d}{\sim}\mathcal{N}(0,\sigma_{\mu_P,1}^2(j))\ \ \text{and}\ \ T_{\mu_P,2}(j)\stackrel{d}{\sim}\mathcal{N}(0,\sigma_{\mu_P,2}^2(j))
\end{eqnarray*} with \begin{eqnarray*}
\sigma_{\mu_P,1}^2(j)&=&\frac{1}{(P(x_j))^2} \biggr[ \sum_{k=1}^{j}P(x_k)(1-P(x_k))-\ \ 2\sum_{\begin{tabular}{ c}
$(k,k')\in [1,j]^2$\\
$k\neq k'$\end{tabular}}
P(x_k)P(x_{k'})\biggr]\\
\text{and} \ \ \sigma_{\mu_P,2}^2(j)&=&\frac{1-P(x_j)}{( P(x_j))^3}\biggr[\sum_{k=1}^{j} P(x_k)\biggr]^2.
\end{eqnarray*}
\begin{proposition}\label{propast}For $j\in [1,r-1]$
, let $
\mu_P^{(n)}(x_j)$ defined by \eqref{meanp}. Then  the following asymptotic results hold
\begin{eqnarray}
&&\limsup_{n\rightarrow+\infty}\frac{|\mu_P^{(n)}(x_j)-\mu_P(x_j)| }{a_n(P)}\leq  A_{\mu_P}(x_j),\ \ \text{a.s.}\\
&&\sqrt{n}(\mu_P^{(n)}(x_j)-\mu_P(x_j))\stackrel{\mathcal{D} }{\rightsquigarrow} \mathcal{N}(0,
\sigma_{\mu_P}^2(j)),\ \ \text{as}\ \ n\rightarrow+\infty.\end{eqnarray}
\end{proposition}

\subsection{Asymptotic behavior of discrete residual/past inaccuracy measures estimators} $\,$\\

\noindent Given $j\in [1,r-1]$, we establish asymptotic limits for $\mathcal{R}_{(X,Y)}^{(n)}(x_j)$ and for $\mathcal{P}_{(X,Y)}^{(n)}(x_j)$.\\

\noindent  For a fixed $j\in [1,r-1]$, let 
\begin{equation}\label{inmxy}
\mathcal{R}_{(X,Y)}^{(n)}(x_j)=-\sum_{k=j}^{r} \frac{\widehat{p}_{n}^{(k)}}{\overline{P}_n(x_j)}\log \frac{q_k}{\overline{Q}(x_j)}
\end{equation}and denote
\begin{eqnarray*}
A_{R_{(X,Y)}}(x_j)&=&\sum_{k=j}^{r} \left\vert\log \frac{q_k}{\overline{Q}(x_j)}\right\vert\\
 \sigma_{R_{(X,Y)}}^2(x_j)&=& \sigma_{RI,1}^2(j)+\sigma_{RI,2}^2(j)+2\,\text{Cov}\left(T_{RI,1}(j),T_{RI,2}(j)\right),
\end{eqnarray*}where $T_{RI,1}(j)\stackrel{d }{\sim}\mathcal{N}\left(0, \sigma_{RI,1}^{2}(j)\right) $ and $T_{RI,2}(j)\stackrel{d }{\sim}\mathcal{N}\left(0, \sigma_{RI,2}^{2}(j)\right) $ with 
 \begin{eqnarray*}
\sigma_{RI,1}^{2}(j)&=&\frac{1}{(\overline{P}(x_j))^2 }\biggr(
   \sum_{k=j}^{r}p_k(1-p_k)\biggr[\log \frac{q_k}{\overline{Q}(x_j)}\biggr]^2\\
  & &\ \ \ \ \ \ \ \  -\ \ \  2\ \ \sum_{%(k,k')\in [1,r-1]^2,
  k\neq k'
  }p_kp_{k'}\log \frac{q_k}{\overline{Q}(x_j)}\log \frac{q_{k'}}{\overline{Q}(x_j)}
  \biggr)\\
  \text{and}\ \ \sigma_{RI,2}^2(j)&=&\frac{1-\overline{P}(x_j)}{( \overline{P}(x_j))^3}\biggr[\sum_{k=j}^{r} p_k\left(\log \frac{q_k}{\overline{Q}(x_j) }\right)\biggr]^2.
 \end{eqnarray*}
\begin{proposition} \label{corinacm} For a fixed $j\in [1,r-1]$, let $
\mathcal{R}_{(X,Y)}^{(n)}(x_j)$ defined by \eqref{inmxy}.  Then the following asymptotic results hold
\begin{eqnarray}\label{rinacmeas}
&& \limsup_{n\rightarrow+\infty}\frac{ \left\vert  \mathcal{R}_{(X,Y)}^{(n)}(x_j)-\mathcal{R}_{(X,Y)}(x_j))\right \vert }{ a_{R,n}(p)}\leq A_{R_{(X,Y)}}(x_j),\ \ \text{a.s.},\\
&&\label{rinacmean}  \sqrt{n}( \mathcal{R}_{(X,Y)}^{(n)}(x_j)-\mathcal{R}_{(X,Y)}(x_j))\stackrel{\mathcal{D}}{ \rightsquigarrow} \mathcal{N}\left(0,\sigma_{R_{(X,Y)}}^2(x_j) \right),\ \ \text{as}\ \ n\rightarrow+\infty.
  \end{eqnarray}

\end{proposition}
\begin{proof}
$\,$ \\

\noindent  For a fixed $j\in [1,r-1]$, we have  
\begin{eqnarray*}
\mathcal{R}_{(X,Y)}^{(n)}(x_j)-
\mathcal{R}_{(X,Y)}(x_j)&=&-\sum_{k=j}^{r}\Delta_{R,n}(p_k) \log \frac{q_k}{\overline{Q}(x_j)},
\end{eqnarray*}where $
  \Delta_{R,n}(p_k)$ is given by \eqref{Deltko}.\\
  
  \noindent Therefore
\begin{eqnarray*}
&& \limsup_{n\rightarrow+\infty}\frac{ \left\vert  \mathcal{R}_{(X,Y)}^{(n)}(x_j)-\mathcal{R}_{(X,Y)}(x_j))\right \vert }{ a_{R,n}(p)}\leq   \sum_{k=j}^{r}\left\vert \log \frac{q_k}{\overline{Q}(x_j)} \right\vert,\ \ \text{a.s.}
\end{eqnarray*}which proves the claim \eqref{rinacmeas}.\\

\noindent Let prove the claim \eqref{rinacmean}. We have \begin{eqnarray*}
\sqrt{n}\left( \mathcal{R}_{(X,Y)}(x_j)-
\mathcal{R}_{(X,Y)}^{(n)}(x_j)\right) &=&-\sum_{k=j}^{r}\sqrt{n}\Delta_{R,n}(p_k) \log \frac{q_k}{\overline{Q}(x_j)}  
\end{eqnarray*}so that, using the same technique as in the proof of the Proposition \ref{resient}, we conclude that 
\begin{eqnarray*}
\sqrt{n}\left( \mathcal{R}_{(X,Y)}(x_j)-
\mathcal{R}_{(X,Y)}^{(n)}(x_j)\right)\stackrel{\mathcal{D} }{\rightsquigarrow}\mathcal{N}\left(0, \sigma_{R_{(X,Y)}}^2(x_j)\right),
\end{eqnarray*}where 
\begin{equation*}
 \sigma_{R_{(X,Y)}}^2(x_j)=\sigma_{RI,1}^2(j)+\sigma_{RI,2}^2(j)+2\,\text{Cov}\left(T_{RI,1}(j),T_{RI,2}(j)\right),
\end{equation*} 
where $T_{R,1}(j)\stackrel{d }{\sim}\mathcal{N}\left(0, \sigma_{R,1}^{2}(j)\right) $ and $T_{R,2}(j)\stackrel{d }{\sim}\mathcal{N}\left(0, \sigma_{R,2}^{2}(j)\right) $.\\

\noindent Which confirms the claim \eqref{rinacmean} and ends the proof of the proposition.
\end{proof}

\bigskip \noindent 
The following proposition concerns the almost sure converge and the asymptotic normality of the discrete past inaccuracy measure estimator between $X$ and $Y$. \\

\noindent  For a fixed $j\in [1,r]$, let  
\begin{eqnarray}\label{inpmxy}
 \mathcal{P}_{(X,Y)}^{(n)}(x_j)=-\sum_{k=1}^{j} \frac{\widehat{p}_{n}^{(k)}}{P_n(x_j)}\log \frac{q_k}{Q(j)}
\end{eqnarray}
and denote
\begin{eqnarray*}
A_{P_{(X,Y)}}(x_j)&=&\sum_{k=1}^{j} \left\vert\log \frac{q_k}{Q(j)}\right\vert\\
\sigma_{P_{(X,Y)}}^2(x_j)&=& \sigma_{PI,1}^2(j)+\sigma_{PI,2}^2(j)+2\,\text{Cov}\left(T_{PI,1}(j),T_{PI,2}(j)\right),
\end{eqnarray*}
where $T_{PI,1}(j)\stackrel{d }{\sim}\mathcal{N}\left(0, \sigma_{PI,1}^{2}(j)\right) $ and $T_{PI,2}(j)\stackrel{d }{\sim}\mathcal{N}\left(0, \sigma_{PI,2}^{2}(j)\right) $ with 
 \begin{eqnarray*}
\sigma_{PI,1}^{2}(j)&=&\frac{1}{(P(x_j))^2 }\biggr(
   \sum_{k=1}^{j}p_k(1-p_k)\biggr[\log \frac{q_k}{Q(x_j)}\biggr]^2\\
  & &\ \ \ \ \ \ \ \  -\ \ \  2\ \ \sum_{%(k,k')\in [1,r-1]^2,
  k\neq k'
  }p_kp_{k'}\log \frac{q_k}{Q(x_j)}\log \frac{q_{k'}}{Q(x_j)}
  \biggr)\\
  \text{and}\ \ \sigma_{PI,2}^2(j)&=&\frac{1-P(x_j)}{(P(x_j))^3}\biggr[\sum_{k=1}^{j} p_k\left(\log \frac{q_k}{Q(x_j) }\right)\biggr]^2,\ \ (j\neq r).
 \end{eqnarray*}
\begin{proposition} \label{copinacm} For a fixed $j\in [1,r]$, let $
\mathcal{P}_{(X,Y)}^{(n)}(x_j)$ defined by \eqref{inpmxy}. Then the following asymptotic results hold
\begin{eqnarray}\label{pinacmeas}
&& \limsup_{n\rightarrow+\infty}\frac{ \left\vert  \mathcal{P}_{(X,Y)}^{(n)}(x_j)-\mathcal{P}_{(X,Y)}(x_j))\right \vert }{ a_{P,n}(p) }\leq A_{(X,Y)}^{(P)}(x_j),\ \ \text{a.s.},\\
&&\label{pinacmean}  \sqrt{n}( \mathcal{P}_{(X,Y)}^{(n)}(x_j)-\mathcal{P}_{(X,Y)}(x_j))\stackrel{\mathcal{D}}{ \rightsquigarrow} \mathcal{N}\left(0,\sigma_{P_{(X,Y)}}^2(x_j) \right),\ \ \text{as}\ \ n\rightarrow+\infty,
  \end{eqnarray}where 
 $  a_{P,n}(p)$ is given by \eqref{appn}.

\end{proposition}

\bigskip \noindent The proof is similar to that of Proposition \ref{corinacm}. Hence omitted.
\\

\subsection{Asymptotic behavior of the discrete cumulative residual/past inaccuracy measures}$\,$\\

\bigskip \noindent 
The following proposition concerns the almost sure converge and the asymptotic normality of the discrete cumulative residual inaccuracy measure estimator between $X$ and $Y$.
\\

\noindent Let 
\begin{equation}\label{crines}
\mathcal{CR}_{(X,Y)}^{(n)}=-\sum_{j=1}^{r-1} \overline{P}_n(x_j)\log \overline{Q}(x_j)
\end{equation} 
  and denote
  \begin{eqnarray*}
  A_{CR_{(X,Y)}}&=&\sum_{j=1}^{r-1}\left\vert \log \overline{Q}(x_j)\right\vert,\\
  \sigma_{CR_{(X,Y)}}^2&=&\sum_{j=1}^{r-1} \overline{P}(x_j)(1-\overline{P}(x_j))\left(\log \overline{Q}(x_j)\right) ^2\\
 &&\ \ \ \ -\ \ 2\sum_{
\begin{tabular}{ c}
$(j,j')\in [1,r-1]^2$\\
$j\neq j'$\end{tabular} }\overline{P}(x_j)\overline{P}(x_{j'})\left(\log \overline{Q}(x_j)\right)\left(\log \overline{Q}(x_{j'})\right).
  \end{eqnarray*}
 \begin{proposition}
 Let $
\mathcal{CR}_{(X,Y)}^{(n)}$ defined by \eqref{crines}.
 Then the following asymptotic results hold
\begin{eqnarray}
&& \limsup_{n\rightarrow+\infty}\frac{ \left\vert \mathcal{CR}_{(X,Y)}^{(n)}-\mathcal{CR}_{(X,Y)}\right \vert }{ a_n(P) }\leq  A_{CR_{(X,Y)}},\ \ \text{a.s.},\\
&& \sqrt{n}( \mathcal{CR}_{(X,Y)}^{(n)}-\mathcal{CR}_{(X,Y)})\stackrel{\mathcal{D}}{ \rightsquigarrow} \mathcal{N}\left(0,
  \sigma_{CR_{(X,Y)}}^2 \right),\ \ \text{as}\ \ n\rightarrow+\infty.
  \end{eqnarray}

 \end{proposition}
 
\bigskip \noindent The proof is similar to that of Proposition \ref{corinacm}. Hence omitted.
\\

  \bigskip\noindent The following proposition concerns the almost sure converge and the asymptotic normality of the discrete cumulative past inaccuracy measure estimator between $X$ and $Y$.
\\
 \begin{equation}\label{cpines}
\mathcal{CP}_{(X,Y)}^{(n)}=-\sum_{j=1}^{r} P_n(x_j)\log Q(x_j)
\end{equation}
and denote
 \begin{eqnarray*}
  A_{CP_{(X,Y)}}&=&\sum_{j=1}^{r}\left\vert \log Q(x_j)\right\vert\\
  \sigma_{CP_{(X,Y)}}^2&=&\sum_{j=1}^{r} P(x_j)(1-P(x_j))\left(\log Q(x_j)\right) ^2\\
 &&\ \ \ \ -\ \ 2\sum_{
\begin{tabular}{ c}
$(j,j')\in [1,r]^2$\\
$j\neq j'$\end{tabular} }P(x_j)P(x_{j'})\left(\log Q(x_j)\right)\left(\log Q(x_{j'})\right).
  \end{eqnarray*}
  \begin{proposition}
 Let $\mathcal{CP}_{(X,Y)}^{(n)}$ defined by \eqref{cpines}.  Then the following asymptotic results hold
\begin{eqnarray}
&& \limsup_{n\rightarrow+\infty}\frac{ \left\vert \mathcal{CP}_{(X,Y)}^{(n)}-\mathcal{CP}_{(X,Y)}\right \vert }{ a_n(P) }\leq  A_{CP_{(X,Y)}},\ \ \text{a.s.},\\
&& \sqrt{n}( \mathcal{CP}_{(X,Y)}^{(n)}-\mathcal{CP}_{(X,Y)})\stackrel{\mathcal{D}}{ \rightsquigarrow} \mathcal{N}\left(0,
  \sigma_{CP_{(X,Y)}}^2 \right),\ \ \text{as}\ \ n\rightarrow+\infty.
  \end{eqnarray}
  \end{proposition}
  
\bigskip \noindent The proof is also similar to that of Proposition \ref{corinacm}. Hence omitted.
\\

\section{Simulation study} \label{simul}

\noindent In this section, we present two examples to demonstrate the consistency and the asymptotic normality of the proposed measures of information estimators developed in the previous sections. 
\begin{example}
$\,$
\end{example}
\noindent For simplicity, let $X$ be a discrete random variable whose probability distribution is that of a discrete Weibull distribution of type $II$ %$\mathcal{W}_2(6,2)$
 with maximum lifetime $r=6$ and shape parameter $\beta=2$.\\
  %and with support $\{1,2,\cdots,r\}$
  This discrete lifetime distribution is  used in reability for modeling discrete lifetimes of components. Its \textit{p.m.f.}  is defined by (see \cite{ceryl})
$$p_k= \frac{k}{6}\prod_{i=1}^{k-1}\left( \frac{6-i}{6}\right),\ \ k=1,2,3,\cdots,6,$$ that is 

\begin{eqnarray*}
&& p_1=1/6,\ \ p_2=5/18,\ \ p_3=5/18,\ \ \\
&& p_4=5/27,\ \ p_5=25/324, \ \ \text{and}\ \ p_6=5/324.
\end{eqnarray*}

\bigskip \noindent \textsc{Table} \ref{tabrp} presents the values of $\mathcal{R}_X(x_j)$, $\mathcal{P}_X(x_j)$, $\mu_R(x_j)$, and of $\mu_P(x_j)$, where $x_j\in[1,6]$. We observe that
as the age $x_j$ increases, $\mathcal{R}_X(x_j)$ decreases while $\mathcal{P}_X(x_j)$ increases.
\noindent Hence $X$ has a decreasing uncertainty of residual lifetime and an increasing uncertainty of past lifetime.\\

\noindent \textsc{Table} \ref{curpe} presents the values of $\mathcal{E}_{Sh}(X)$, $\mathcal{CR}_X$, and $\mathcal{CP}_X$.
 The uncertainty contained in distribution function is lower than that contained in the  survival function which is lower than that contained in the past entropy at times $x_6=6$ since \\
 $\mathcal{CP}_X < \mathcal{CR}_Y<\mathcal{P}_X(6).$

\begin{example}
$\,$
\end{example}
\noindent We suppose that $X$ is the actual random variable corresponding to the observations with outcomes $\{1,2,3,4\}$ and \textit{p.m.f.}'s
\begin{equation}\label{xsim2}
p_1=7/40,\ \ p_2=11/20,\ \ p_3=1/4,\ \ p_4=1/40
\end{equation}
and $Y$ is the random variable assigned by the experimenter with \textit{p.m.f.}'s   
\begin{equation}\label{ysim2}
q_k=\frac{k^3}{100},\ \ k=1,2,3,4.
\end{equation}  This distribution is discussed in
\cite{layw}.\\

\bigskip  
\noindent \textsc{Table} \ref{rpxy}
presents the values of $\mathcal{R}_{(X,Y)}(x_j)$ and $\mathcal{P}_{(X,Y)}(x_j)$, where $x_j\in[1,4]$. \\

\noindent \textsc{Table} \ref{crcpxy} presents the values of $K_{(X,Y)},$  $\mathcal{CR}_{(X,Y)}$, and $\mathcal{CP}_{(X,Y)}$.

\bigskip  \noindent In each \textsc{Figures}  \ref{figres0p5}, \ref{cre1},  \ref{mrpent}, \ref{rpinac1} and \ref{crinac1},  left panels represent the plots of information measure estimator, built from sample sizes of $n=100,\, 200,\cdots,30000,$ and the true information measure (represented by horizontal black line). The middle panels show the histograms of
the data and where the red line represents the plots of the theoretical normal distribution
calculated from the same mean and the same standard deviation of the data. The right
panels concern the Q-Q plot of the data which display the observed values against normally
distributed data (represented by the red line). We see that the underlying distribution of
the data is normal since the points fall along a straight line.

\section{Conclusion}\label{conclusion}

This paper joins a growing body of literature on estimating  residual/past entropies and  inaccuracy measures in the discrete case on finite sets. We adopted the plug-in method and we derived almost sure rates of convergence and asymptotic normality of these measures of uncertainty.

\newpage
\vspace{1ex}
\begin{center}
\vspace{1ex} 
\begin{table}
\centering
\begin{tabular}{ |c| ccc c c c| }
\hline
%&&&&&&\\
$x_j$ & $1$&$2$ &$3$&$4$&$5$ &6\\
%&&&&&&\\
\hline
&&&&&&\\
$\mathcal{R}_X(x_j)$ &  1.682734& 1.433071
& 1.192166 & -0.9357332& -8.04719&$\times$
\\
$\mathcal{P}_X(x_j)$ & 0
& 0.6615632 & 1.073394& 1.360343& 1.528503& 1.5846\\
$\mu_R(x_j)$& 2.12963& 1.694444 & 1.388889 & 1.166667&-1&$\times$\\
$\mu_P(x_j)$&1
& 1.375& 1.846154 & 2.469388 & 3.275862& 4.225309
\\
\hline
\end{tabular}
\vspace{2ex}
\caption{Discrete residual/past entropies and mean residual/past lifetime values at time $x_j$.
The distribution of $X$ being that of the \textit{Weibull} distribution of type $II$ with maximum lifetime $r=6$ and shape parameter $\beta=2$.}\label{tabrp}
\end{table}
\vspace{1ex}
\begin{table}
\vspace{1ex}
\begin{tabular}{|ccc|}
\hline
Shannon entropy &Cumulative residual entropy & Cumulative past entropy\\
$\mathcal{E}_{Sh}(X)=\mathcal{P}_X(6)
$&$\mathcal{CR}_{X}$&$\mathcal{CP}_X$\\
&&\\
$ 1.5846 \,\text{nats}$
&$ 1.118998 \ \text{nats}$ &$ 0.9975468\ \text{nat}$\\
\hline
\end{tabular}
\vspace{2ex}
\caption{Discrete cumulative residual/past entropies values. The distribution of $X$ being that of the \textit{Weibull} distribution of type $II$ with maximum lifetime $r=6$ and shape parameter $\beta=2$.}\label{curpe}
\end{table}
\end{center}

 \vspace{1ex}
 \begin{table} 
 %\centering
 \begin{tabular}{ |c|cccc|}
 \hline
 $x_j$ & $1$ & $2$&$3$ &$4$\\
 &&&&\\
 $\mathcal{R}_{(X,Y)}(x_j)$&3.058783&5.9994&8.630462
&$\times$\\
 $\mathcal{P}_{(X,Y)}(x_j)$ &
   0&0.6197172& 1.565414& 2.5335460
\\
\hline
  \end{tabular} 
  \vspace{2ex}
  \caption{Computations of 
   $\mathcal{R}_{(X,Y)}(x_j)$ and $\mathcal{P}_{(X,Y)}(x_j)$.}\label{rpxy}
 \end{table}

\begin{table}
\begin{tabular}{|ccc|}
\hline
Inaccuracy measure &C.R. inacc.  meas.&C.P inacc.  meas.\\
$K_{(X,Y)}=\mathcal{P}_{(X,Y)}(4)$ &$\mathcal{CR}_{(X,Y)}$&$\mathcal{CP}_{(X,Y)}$\\
&&\\
 2.5335460 nats& 0.04538414 nat
& 3.547775 nats
\\
\hline
\end{tabular}\vspace{1ex}
  \caption{Computations of $\mathcal{CR}_{(X,Y)}$ and $\mathcal{CP}_{(X,Y)}$.}\label{crcpxy}
\end{table}

\newpage

\begin{figure}
[H]
\includegraphics[scale=0.25]{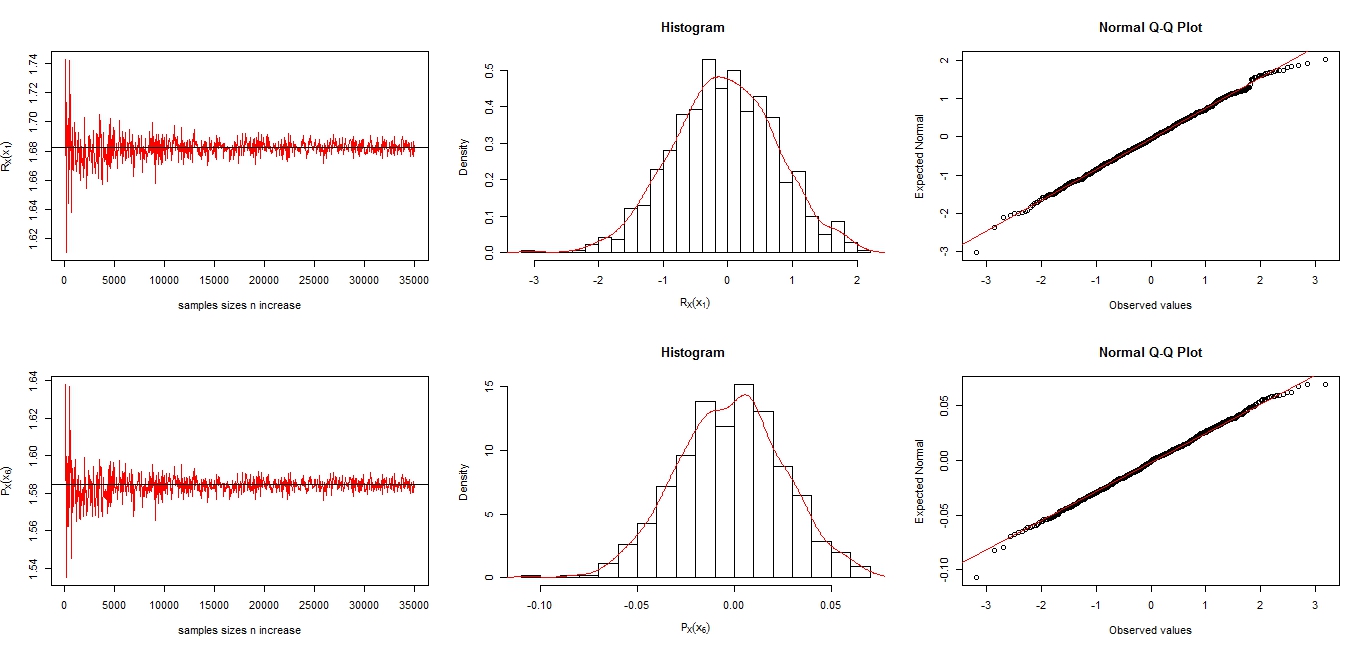} 
  \caption{Plots of $\mathcal{R}_X^{(n)}(1)$ and $\mathcal{P}_X^{(n)}(6)$ when samples sizes increase, histograms and normal Q-Q plots  versus $\mathcal{N}(0,1)$. The distribution of $X$ being that of the \textit{Weibull} distribution of type $II$ with maximum lifetime $r=6$ and shape parameter $\beta=2$.
}\label{figres0p5}
\end{figure}

\begin{figure}[H]
\includegraphics[scale=0.25]{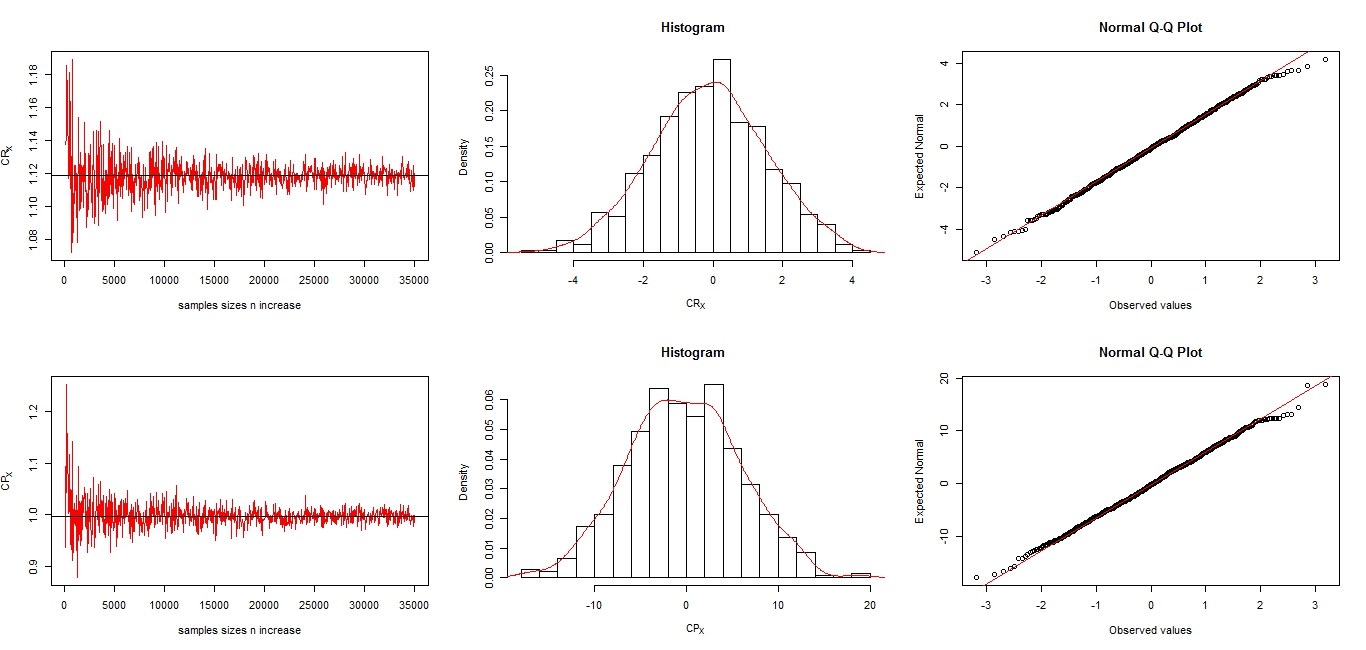} 
\caption{Plots of $ \mathcal{CR}_X^{(n)}$ and $ \mathcal{CP}_X^{(n)}$ when samples sizes increase, histograms and normal Q-Q plots  versus $\mathcal{N}(0,1)$.  The distribution of $X$ being that of the \textit{Weibull} distribution of type $II$ with maximum lifetime $r=6$ and shape parameter $\beta=2$.
}\label{cre1}
\end{figure}

\begin{figure}[H]
 \includegraphics[scale=0.25]{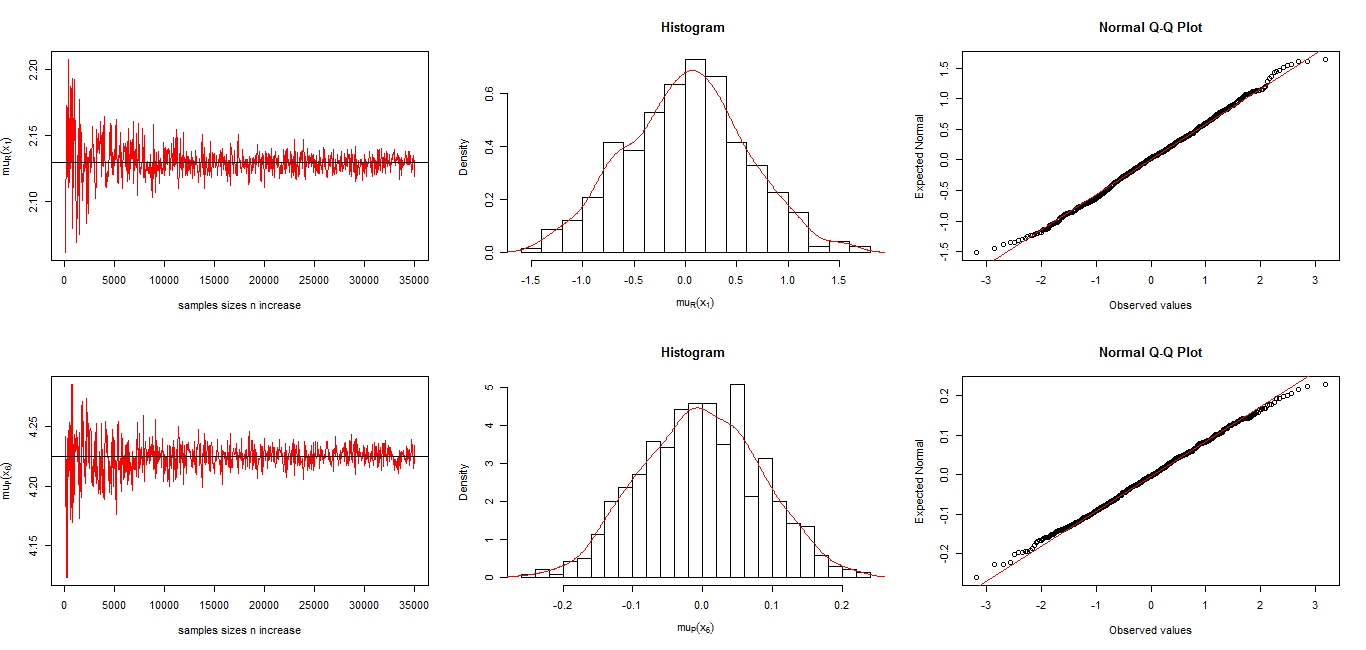} 
\caption{Plots of $\mu_R^{(n)}(1)$ and $\mu_P^{(n)}(6)$  when samples sizes increase, histograms and normal Q-Q plots  versus $\mathcal{N}(0,1)$.  The distribution of $X$ being that of the \textit{Weibull} distribution of type $II$ with maximum lifetime $r=6$ and shape parameter $\beta=2$.
}\label{mrpent}
\end{figure}

\begin{figure}[H]
\includegraphics[scale=0.25]{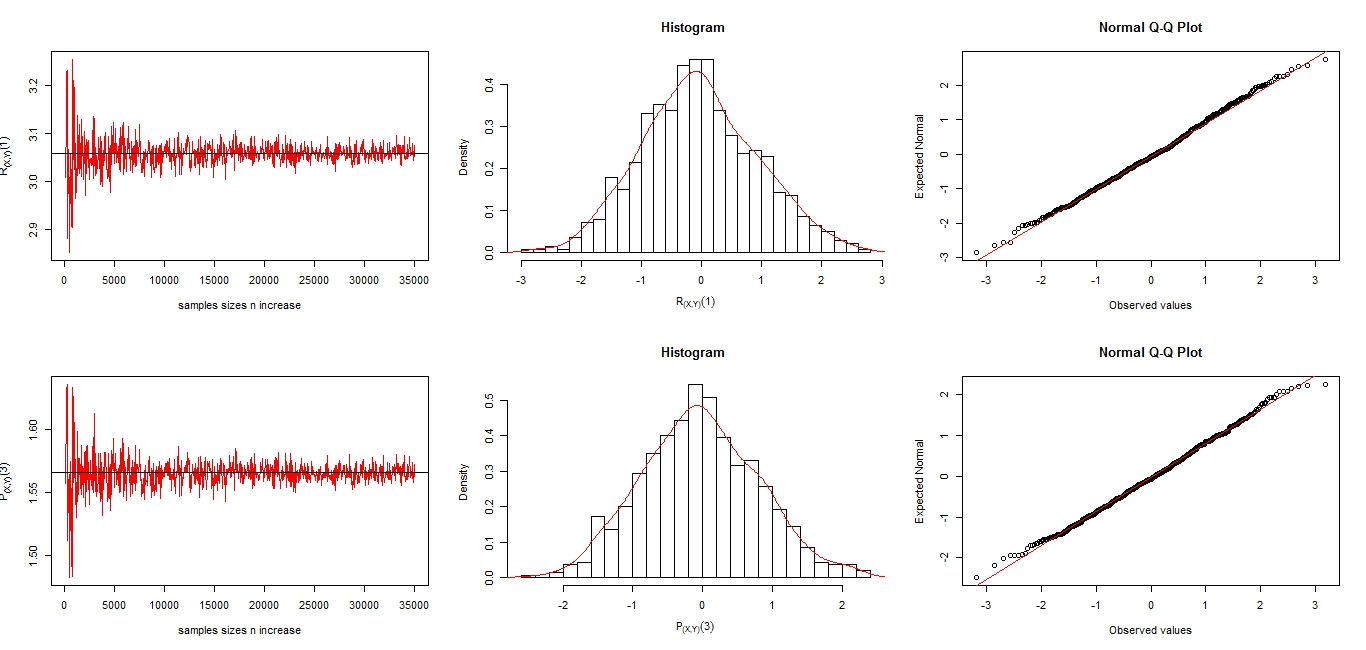} 
\caption{Plots of $\mathcal{R}_{(X,Y)}^{(n)}(1)$ and $\mathcal{P}_{(X,Y)}^{(n)}(3)$  when samples sizes increase, histograms and normal Q-Q plots  versus $\mathcal{N}(0,1)$. $X$ and $Y$ being two random variables which \textit{p.m.f.}'s are given respectively by \eqref{xsim2} and \eqref{ysim2}.
}\label{rpinac1}
\end{figure}
\begin{figure}[H]
\includegraphics[scale=0.25]{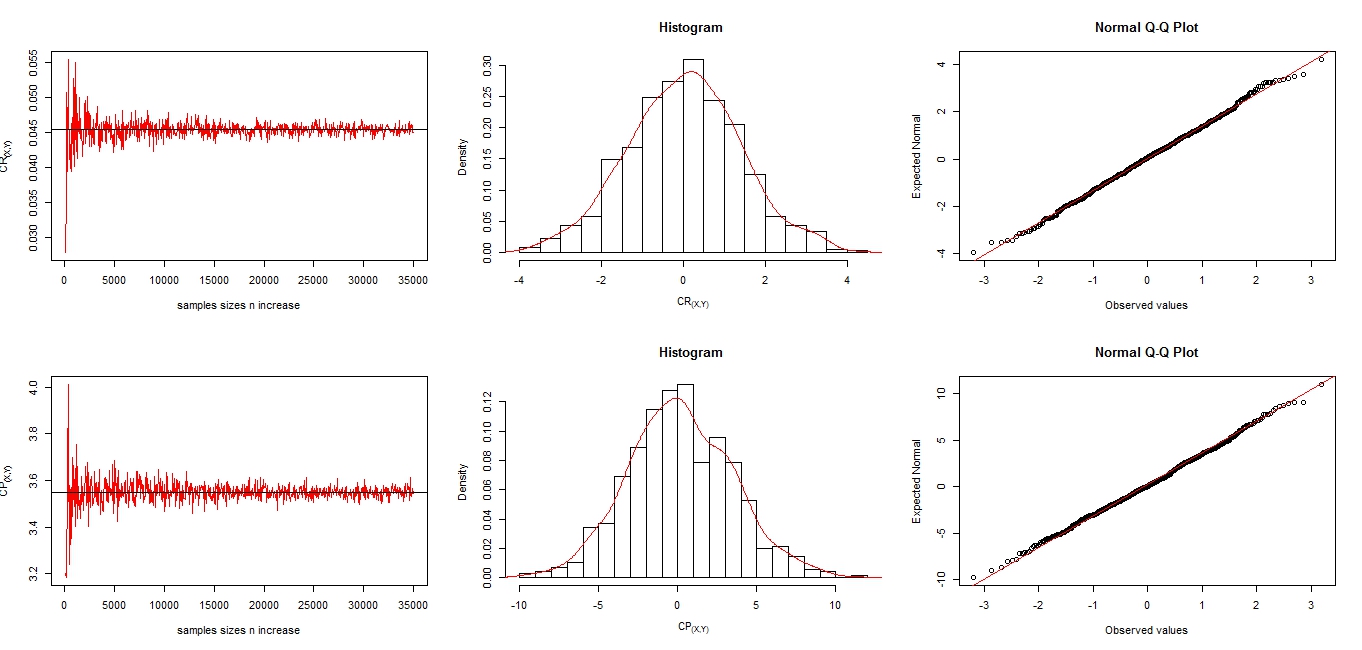}  
\caption{Plot of $\mathcal{CR}_{(X,Y)}^{(n)}$ and $\mathcal{CP}_{(X,Y)}^{(n)}$  when samples sizes increase, histograms and normal Q-Q plots  versus $\mathcal{N}(0,1)$. $X$ and $Y$ being two random variables which \textit{p.m.f.}'s are given respectively by \eqref{xsim2} and \eqref{ysim2}.
}\label{crinac1}
\end{figure}

\end{document}